\documentclass[preprint,11pt,authoryear]{article}
\usepackage[utf8]{inputenc}

\usepackage[top=2.5cm,bottom=2.5cm,left=3cm,right=2.5cm]{geometry}
\usepackage[linkcolor = red, colorlinks = true]{hyperref}
\usepackage[dvips]{epsfig}
\usepackage{color}
\usepackage{tikz-cd}
\usepackage{amsgen, amstext,amsbsy,amsopn,amssymb, amsthm,mathrsfs,mathptmx,amsfonts,amssymb,amscd,amsmath,euscript,verbatim,calc,multicol,mathtools}
\usepackage{url}
\usepackage{hyperref}
\usepackage[inline]{enumitem}   
\usepackage{authblk}
\usepackage[multiple]{footmisc}
\usepackage{bigfoot}
\newcommand\blfootnote[1]{%
  \begingroup
  \renewcommand\thefootnote{}\footnote{#1}%
  \addtocounter{footnote}{-1}%
  \endgroup
}
\usepackage{soul,xcolor}
\theoremstyle{plain}
\newtheorem{theorem}{Theorem}[section]

\newtheorem{corollary}[theorem]{Corollary}
\newtheorem{lemma}[theorem]{Lemma}

\theoremstyle{definition}
\newtheorem{definition}[theorem]{Definition}

\theoremstyle{remark}
\newtheorem{remark}[theorem]{\bf Remark}
\begin{document}    
\author[1]{Ambily Ambattu Asokan}
\author[2]{Sugilesh H.\footnote{Corresponding author}}
\affil[1,2]{Department of Mathematics,
  \textit{Cochin University of Science and Technology},
  \textit{Kochi-682022, Kerala, India.}}
  \affil[]{\textsuperscript{}\emph{ambily@cusat.ac.in}, \textsuperscript{2}\emph{sugileshh@gmail.com}}
 \title{Horrocks' theorem for odd orthogonal groups}
  \date{}
\maketitle    
\begin{abstract}
 We prove Horrocks' theorem for the odd elementary orthogonal group, which gives a decomposition of an orthogonal matrix with entries from a polynomial ring $R[X]$, over a commutative ring $R$ in which 2 is invertible, as a product of an orthogonal matrix with entries in $R$ and an elementary orthogonal matrix with entries from $R[X]$. 
\end{abstract}
\blfootnote{Date:\date{13/03/2025}}
\section{Introduction}

In 1955, J-P. Serre posed a fundamental question in algebra, asking whether every finitely generated projective module over the polynomial ring $R = k[X_1, ,\dots, X_r]$ is free \cite{JPS1955}. This question, known as `Serre's Conjecture', is equivalent to determining whether all algebraic vector bundles over affine space $k^r$ are trivial. This problem remained unsolved for many years until D. Quillen and A. A. Suslin independently gave a complete solution in 1976. Quillen's solution \cite{Quillen1976} relied significantly on Horrocks' theorem, while Suslin's approach  \cite{Suslin1976} was based on the local-global principle. Since then, these two techniques, Local-Global Principle and Horrocks’ Theorem have become central tools in algebraic K-Theory.

\vspace{1mm}

Horrocks' theorem provides a criterion for determining when a vector bundle over the affine line ${\mathbb A}^1$ is trivial \cite{Hor1964}. The algebraic formulation of Horrocks' theorem states the following: ``\emph{Let $R$ be a commutative local ring, and let $P$ be a finitely generated projective module over the polynomial ring $R[X]$. If the localization $P_X$ is free over $R_X$, then $P$ itself must be free over $R[X]$}
" \cite{Horrocks1964}. Some important analogues of Horrocks' Theorem found in the literature include:
\emph{Serre's Splitting Theorem} \cite{JPS1955}, \emph{Quillen-Suslin Theorem} \cite{Quillen1976,Suslin1976}, \emph{Bass-Quillen Conjecture} \cite{Bass1963}, and \emph{Homotopy Lifting Property for Vector Bundles} \cite{Thomason1985}.

In \cite{Suslin1977}, A. A. Suslin extended Horrock's Theorem on projective modules by proving a $\rm{K}_1$-analogue as follows:
\begin{theorem}
     Let R be a commutative ring and $\alpha \in
{\rm GL}_n(R[X]), n \geq 3$. Suppose that there exists $\beta \in {\rm GL}_n(R[X^{-1}])$ such that
$\alpha \cdot{\beta}^{-1} \in {\rm E}_n(R[X,X^{-1}])$. Then $\alpha \in {\rm GL}_n(R)\cdot{\rm E}_n(R[X])$.
\end{theorem} 
In \cite{SusKop1977}, A. A. Suslin and V. I. Kopeiko established the even orthogonal analogue of Horrocks' theorem as follows:  \begin{theorem}
        Let R be a commutative ring with unity in which 2 is invertible, $n\geq3$ is a natural number, $\alpha \in {\rm O}_{2n}(R[X])$, $\beta \in {\rm O}_{2n}(R[X^{-1}])$. If $\alpha \cdot \beta^{-1}\in {\rm EO}_{2n}(R[X,X^{-1}]) $, then $\alpha \in {\rm O}_{2n}(R)\cdot{\rm EO}_{2n}(R[X])$ and $\beta \in {\rm O}_{2n}(R) \cdot {\rm EO}_{2n}(R[X^{-1}])$.
   \end{theorem}

In \cite{KOPK1978}, V. I. Kopeiko proved the symplectic group analogue of Horrocks' theorem as follows:
\begin{theorem}
    Let R be a commutative ring with unity in which 2 is invertible and $\alpha \in
{\rm Sp}_{2n}(R[X])$, $n \geq 1$. Suppose that there exists a symplectic transformation $\beta \in {\rm Sp}_{2n}(R[X^{-1}])$ such that
$\alpha \cdot{\beta}^{-1} \in {\rm ESp}_{2n}(R[X,X^{-1}])$, then $\alpha \in {\rm Sp}_{2n}(R) \cdot {\rm ESp}_{2n}(R[X])$.
\end{theorem} 

In \cite{Fasel2015}, B. Calmès and J. Fasel provided the generators of the odd elementary orthogonal group, giving insights into the algebraic structure of these groups. We establish the odd orthogonal analogue of Horrocks' theorem. Furthermore, we derive several  decomposition results for the elementary orthogonal group. 

The main result we prove in this paper is as follows:   \begin{theorem}
       Suppose $R$ is a commutative ring in which 2 is invertible (need not be local). Let $\alpha \in {\rm O}_{2n+1}(R[X])$ and $\beta \in {\rm O}_{2n+1}(R[X^{-1}])$, for $n\geq 3$. If $\alpha \cdot \beta^{-1}\in {\rm EO}_{2n+1}(R[X,X^{-1}])$, then $\alpha \in {\rm O}_{2n+1}(R)\cdot{\rm EO}_{2n+1}(R[X])$ and $\beta \in {\rm O}_{2n+1}(R) \cdot {\rm EO}_{2n+1}(R[X^{-1}])$.
   \end{theorem}
   We prove the odd orthogonal analogue of Horrocks's theorem extending the classical result to the setting of odd orthogonal groups. Proving this result is particularly challenging because of the complexity of the computations and structure. The computer algebra system GAP (Groups, Algorithms, Programming) played a crucial role in performing further calculations and lead us to define specific subgroups of the odd orthogonal group, such as ${\rm TO}_{2n+1}(R)$ and ${\rm TO}'_{2n+1}(R).$ Certain results that hold trivially in the even orthogonal group require deeper analysis in the odd orthogonal group. This article discusses computational strategies, and explores the structural distinction between the even and odd orthogonal groups.

   \section{Preliminaries}
Let $R$ denote a commutative ring with unity in which 2 is invertible. Let ${\rm M}_n(R)$ denote the set of all $n\times n$ matrices with entries in $R$ and let 
${\rm GL}_n(R)$ denote the group of all $n \times n$ invertible matrices with entries in $R$. If $I$ is an ideal of a ring $R$, then ${\rm M}_{n,r}(I)$ denotes the set of all $n\times r$ matrices with entries from $I$. Let ${\rm Alt}_n(R)$ denote the subset of ${\rm M}_n(R)$ consisting of the matrices $\alpha$ such that $\alpha_{ii}=0$ and $\alpha_{ij}=-\alpha_{ji}$ if $i\neq j$. The subgroup ${\rm GL}_n(I)$ of the group ${\rm GL}_n(R)$, consists of matrices of the form $I_n+\alpha$, where $\alpha \in {\rm M}_{n}(I)$. For $\alpha$, $\beta\in {\rm M}_n(R)$, we follow the notation $[\alpha]\perp[\beta]$ and $[\alpha]\top[\beta]$ for $\begin{pmatrix}
    \alpha&0\\0&\beta
\end{pmatrix}$ and $\begin{pmatrix}
    0&\alpha\\\beta&0
\end{pmatrix}$ respectively.

\begin{definition}[{\emph {Quadratic form}}]A quadratic form on an $R$-module $M$ is a map $q : M \rightarrow R$ such that
        \begin{enumerate}[label=(\roman*)]
         \item $q(ax) = a^2 q(x)$, $\forall a \in R$, $x \in M$
        \item  The map $B_q : M \times M \rightarrow R$ defined by
$B_q(x, y) = q(x + y) - q(x) - q(y)$ is bilinear.
 \end{enumerate} 
 \end{definition}
The map $B_q$ is called the bilinear form associated with $q$.
If 2 is invertible, then $q(x)=\frac{1}{2}
B_q(x, x)$, for all $x \in M$. The pair $(M, q)$ is called the quadratic
$R$-module. A quadratic form $q$ is said to be non-degenerate if the associated bilinear form  $B_q$ is non-degenerate. That is, if the mapping from $M$ to its dual  $M^*$ defined by $x\mapsto B_q(x,y)$ is an isomorphism. A quadratic space over $R$ is a pair $(M,q)$, where $M$ is a finitely generated projective $R$-module and $q : M \rightarrow R$ is a non-degenerate quadratic form. Let $(M_1, q_1)$ and $(M_2, q_2)$ be two quadratic $R$-modules, then the orthogonal sum of these two quadratic $R$-modules is defined as $(M_1, q_1) \perp (M_2, q_2) = (M_1 \oplus M_2, q)$; 
where $q(x_1 + x_2) = q_1(x_1) + q_2(x_2)$.
\begin{definition}[{\emph{Hyperbolic Space}}]  Let $P$ be a finitely generated projective $R$-module. Then the hyperbolic space $\mathbb{H}(P)$ is the quadratic space $P\oplus P^*$ with the bilinear form $\langle(x_1,f_1),(x_2,f_2)\rangle :=f_2(x_1)+f_1(x_2)$, for $x_1,x_2\in P$ and $f_1,f_2 \in P^*$, where $P^*$ is the dual space of $P$.
\end{definition}

Let $(M_1, q_1)$ and $(M_2, q_2)$ be two quadratic $R$-spaces. We call $\alpha:M_1\rightarrow M_2$ an {\it isometry}, if $q_1(m)=q_2(\alpha(m)).$ For a quadratic space $(M,q)$, let ${\rm O}(M,q)$ denote the group of all isometries of $(M,q)$ onto itself and let $\tilde{O}(M,q)$ denote the group of all linear automorphisms of $M$, sending the form $q$ into the form $a\cdot q$ for an invertible $a\in R.$ If $M$ is an $R$- module and $v\in M$, we put $o(v)=\{ \phi(v): \phi\in {\rm Hom}_R(M,R)\}$, which is an ideal of $R$ called the {\it order ideal}. 

\vspace{1mm}

If $(M,q)$ is a quadratic $R$-space and let $u,v\in M$. If $\phi$  is the bilinear form associated to $q$ such that $q(u)$, $\phi(u,v)=0$. For $a\in R$, consider the map ${\rm E}_{u,v}^o(a)$ given by,
$${\rm E}_{u,v}^o(a)(w)=w+(u\cdot\phi(v,w)-v\cdot\phi(u,w))\cdot a-u\cdot a^2\cdot\phi(u,w)\cdot q(v).$$
 For a subset $V$ of $M$, we define $V^\perp=\{m\in M : \phi(u,m)=0\ \forall\ u\in V\}$.

\begin{lemma}[{\rm \cite[Lemma~1.7]{SusKop1977}}]\label{Pre2.3}The maps ${\rm E}_{u,v}^o(a)$ satisfies the following.
   \begin{enumerate}[label={\emph{(\roman*)}}]
       \item ${\rm E}_{u,v}^o(a)\in {\rm O}(M,q)$ and ${\rm E}_{u,u}^o(a)=1_M$.
       \item ${\rm E}_{u,v}^o(a)={\rm E}_{u\cdot a,v}^o(1)={\rm E}_{u,v\cdot a}^o(1)$.
       \item The mapping $v\mapsto {\rm E}_{u,v}^o$ is a homomorphism of the additive group of $v^\perp$ into ${\rm O}(M,q)$, where $u$ is isotropic.
       \item If $L$ is totally isotropic subspace of $M$, then the mapping $(u,v)\mapsto {\rm E}_{u,v}^o$ of $L\times L$ into ${\rm O}(M,q)$ is alternating and bilinear, and if $L\subset w^\perp $ for some $w\in M$, then for $u,v\in L$, we have ${\rm E}_{u,v}^o\cdot {\rm E}_{v,w}^o={\rm E}_{u+v,w}^o\cdot {\rm E}_{u,v}^o( q(w))$.
       \item If $\alpha\in \tilde{{\rm O}}(M,q)$ and $\alpha$ sends $q$ to $a\cdot q$, then $\alpha\cdot {\rm E}_{u,v}^o(b)\cdot \alpha^{-1}={\rm E}_{\alpha u,\alpha v}^o(a^{-1}b).$
   \end{enumerate} 
\end{lemma}
        \begin{definition}[{\emph{Orthogonal group}}] 
        The orthogonal group {${\rm O}_{n}(R)$} with respect to an invertible symmetric matrix $\phi$ is defined to be the collection $\{\alpha\in {\rm GL}_{n}(R):\alpha^t\phi \alpha=\phi\}.$ Here, we consider the  symmetric matrix, $\phi=\begin{cases}
             \widetilde{\psi}_r ~&{\rm if}~ n=2r\\
             2\perp \widetilde{\psi}_r &{\rm if} ~n=2r+1
         \end{cases}$, where $\widetilde{\psi}_r=\sum_{i=1}^{r}e_{i,r+i}+\sum_{i=1}^{r}e_{r+i,i}.$
        \end{definition}
\subsection{Elementary orthogonal group}
        \begin{definition}\label{Pre2.5}[{\emph{Even elementary orthogonal group}}] The even elementary orthogonal group, {${\rm EO}_{2n}(R)$} is generated by matrices of the form
         $oe_{ij}(z)=I_{2n}+e_{ij}(z)-e_{\delta(j)\delta(i)}(z)$,
        where $1\leq i,j \leq 2n$, $i\neq j$, and  
            $\delta(i)=\begin{cases}
            i+n~&{\rm if}~1\leq i\leq n\\i-n~&{\rm if}~n+1\leq i\leq 2n
        \end{cases}$, for $z\in R$.
        
        The group ${\rm EO}_{2n}(R)$ is a subgroup of ${\rm O}_{2n}(R).$
        \end{definition}
         \begin{remark}
         If $1\leq i\neq j\leq 2n$ and $z\in R$, then the generators, $oe_{ij}(z)={\rm E}_{e_i,e_{\delta(j)}}(z)$, where $\{e_1,\ldots,e_{2n}\}$ is the standard basis of $R^{2n}.$
         \end{remark}     

     \begin{definition}$(${\emph{Odd elementary orthogonal group} \cite{Fasel2015}}$)$  The odd elementary orthogonal group, ${\rm {EO}}_{2n+1}(R)$ is generated by the matrices of the following types: For $1\leq i\neq j\leq n$ and $z\in R$, 
 \begin{enumerate}[label=(\roman*)]
   \item $F^1_i(z)= I_{2n+1}+e_{1,n+i+1}(z)-e_{i+1,1}(2z)-e_{i+1,n+i+1}(z^2)$,
  \item $F^2_i(z)= I_{2n+1}+e_{1,i+1}(z)-e_{n+i+1,1}(2z)-e_{n+i+1,i+1}(z^2)$,
    \item $F^3_{i, j}(z)= I_{2n+1}+e_{i+1,j+1}(z)-e_{n+j+1,n+i+1}(z)$,
    \item $F^4_{i, j}(z)= I_{2n+1}+e_{i+1,n+j+1}(z)-e_{j+1,n+i+1}(z)$,
     \item $F^5_{i, j}(z)= I_{2n+1}+e_{n+i+1,j+1}(z)-e_{n+j+1,i+1}(z)$.
\end{enumerate}\end{definition}
  These matrices satisfy the relations:
 $F^3_{i, j}(z) = [F^1_i(z),F^2_j(-\frac{1}{2})]$, $F^4_{i, j}(z) = [F^1_j(z),F^1_i(\frac{1}{2})] $, $F^5_{i, j}(z) = [F^2 _j(z),F^2_i(\frac{1}{2})]$. Hence, we have the following lemma.
   \begin{lemma}
       The group ${\rm {EO}}_{2n+1}(R)$ is generated by matrices of the form $F_{i}^1(z)$ and $F_i^2(z)$, where $1 \leq i \leq n$, for $z\in R$.
   \end{lemma}
   
  \begin{remark}
      The group ${\rm EO}_{2n+1}(R)$ is a subgroup of ${\rm O}_{2n+1}(R).$
   In addition, for $1\leq i\neq j\leq n$ and $z\in R$, we have $F_i^1(z)={\rm E}_{e_{i+1},e_1(-1)}^o(z)$, $F_i^2(z)={\rm E}_{e_{n+i+1},e_1(-1)}^o(z)$,
  where $\{e_1,\ldots,e_{2n+1}\}$ is the standard basis of $R^{2n+1}.$
  \end{remark}

\begin{definition}
       For any natural number, $n$ we have a \emph{canonical embedding ${\rm O}_{2n+1}(R)\rightarrow {\rm O}_{2n+3}(R)$}, defined by
       $$ \alpha=\begin{pmatrix}
           \alpha_{11}&\alpha_{12}&\alpha_{13}\\
           \alpha_{21}&\alpha_{22}&\alpha_{23}\\
           \alpha_{31}&\alpha_{32}&\alpha_{33}
           \end{pmatrix}          
           \mapsto
           \begin{pmatrix}
              \alpha_{11}&\alpha_{12}&0&\alpha_{13}&0\\
           \alpha_{21}&\alpha_{22}&O&\alpha_{23}&O\\
           0&O&1&O&0\\
           \alpha_{31}&\alpha_{32}&O&\alpha_{33}&O\\
           0&O&0&O&1
        \end{pmatrix},$$
         where $\alpha_{11}\in \mathrm{M}_{1,1}(R)$, $\alpha_{22},\alpha_{23},\alpha_{32}$, and $\alpha_{33}\in \mathrm{M}_{n,n}(R)$, $\alpha_{12}, \alpha_{13}\in \mathrm{M}_{1,n}(R)$, and $\alpha_{21}, \alpha_{31}\in \mathrm{M}_{n,1}(R)$. Under this embedding, the groups ${\rm EO}_{2n+1}(R)$, ${\rm {O}}_{2n+1}(I)$, ${\rm EO}_{2n+1}(I)$ and ${\rm EO}_{2n+1}(R,I)$ maps into ${\rm EO}_{2n+3}(R)$, ${\rm O}_{2n+3}(I)$, ${\rm EO}_{2n+3}(I)$ and ${\rm EO}_{2n+3}(R,I)$ respectively.
         \end{definition} 
         
         We can embed a matrix $\alpha$ of even dimension into an odd dimensional group by taking $1\perp \alpha.$ 
         
\begin{definition}
The factor group ${\rm KO}_1(R)$ is defined as ${\rm KO}_1(R):=\frac{{\rm O}(R)}{{\rm EO}(R)}$, where ${\rm O}(R)=\displaystyle\lim_{n\to\infty}{\rm O}_{2n+1}(R)$\\[5pt]and ${\rm EO}(R)=\displaystyle\lim_{n\to\infty}{\rm EO}_{2n+1}(R)$. For the relative group ${\rm EO}(R,I)$, define ${\rm KO}_1(R,I):={\rm O}(R)/{\rm EO}(R,I)$, where ${\rm EO}(R,I)=\displaystyle\lim_{n\to\infty}{\rm EO}_{2n+1}(R,I)$.
\end{definition}
      
\subsection {\bf Subgroups of $ {\rm EO}_{2n+1}(R)$ and ${\rm O}_{2n+1}(R)$}

\begin{itemize}   
\item ${\rm O}_{2n+1}(I)={\rm O}_{2n+1}(R)\cap {\rm GL}_{2n+1}(I)={\rm ker }({\rm O}_{2n+1}(R)\rightarrow {\rm O}_{2n+1}(R/I)).$
 
        \item ${\rm EO}_{2n+1}(I)$ denotes the subgroup of ${\rm O}_{2n+1}(R)$, generated by all $F_i^1(z)$ and $F_i^2(z)$ with $z\in I$.
          
       \item  ${\rm EO}_{2n+1}(R,I)$ denotes the subgroup of ${\rm O}_{2n+1}(R)$, generated by all $\alpha\cdot F^1_i(z)\cdot \alpha^{-1}$ and $\alpha\cdot F^2_i(z)\cdot \alpha^{-1}$ with $\alpha\in\ {\rm EO}_{2n+1}(R),\ z\in I$.

   \item Let ${\rm DO}_{2n+ 1}(R)$ be the subgroup of ${\rm O}_{2n+1}(R)$, consisting of the diagonal matrices. Observe that ${diag}(d_0,d_1,\dots ,d_{2n})$ lies in ${\rm O}_{2n+1}(R)$ if and only if $d_{\delta(i)}=d_i^{-1}$, for $1\leq i\leq n$ and ${d_0}^{-1}=d_0$.    
   \item If $\pi$ is a permutation of the set $\{1,2,\ldots,2n,2n+1\}$, we denote the corresponding permutation matrix by $\sigma_\pi$. The matrix $\sigma_\pi$ is orthogonal if and only if $\pi$ commutes with $\delta$ (Definition~\ref{Pre2.5}). Let ${\rm \Pi O}_{2n+1}(R)$ denote the subgroup of ${\rm O}_{2n+1}(R)$ that consists of these permutation matrices. That is, ${\rm \Pi O}_{2n+1}(R) = \{ \sigma_\pi : \pi\cdot\delta=\delta\cdot\pi\}.$
  \item ${\rm MO}_{2n+1}(R)$ denotes the subgroup of ${\rm O}_{2n+1}(R)$ consisting of the monomial matrices.
\end{itemize}

\section{Construction of ${\rm TO}_{2n+1}(R)$ and ${\rm TO'}_{2n+1}(R)$}
In \cite{SusKop1977}, A. A. Suslin and V. I. Kopeiko constructed the subgroups ${\rm TO}_{2n}(R)$ and ${\rm TO'}_{2n}(R)$ of the even elementary orthogonal group ${\rm EO}_{2n}(R)$ to prove some important splitting results leading to Horrocks' theorem. In this section, we provide the detailed construction of the odd orthogonal analogues, ${\rm TO}_{2n+1}(R)$ and ${\rm TO'}_{2n+1}(R)$ using GAP calculations.

\subsection{Generators for ${\rm TO}_6$}
An element in ${\rm TO}_{2n}(R)$ is given by $\begin{pmatrix}
    \gamma & \delta \\
    O & (\gamma^T)^{-1}
\end{pmatrix}$, where $\gamma$ is an upper triangular matrix with $1$'s on the main diagonal such that $\gamma^{-1}\cdot\delta \in {{\rm Alt}}_n(R)$. For example, an arbitrary element in ${\rm TO}_6(R)$ is of the form, 
\begin{align*}
    A=&\begin{pmatrix}
    \gamma & \delta \\
    O & (\gamma^T)^{-1}
\end{pmatrix} =
\begin{pmatrix}
    1 & a & b & b\cdot q+a\cdot p-a\cdot c\cdot q & b \cdot r + c \cdot q - p & -a \cdot r -q \\
    0 & 1 & c & p & c \cdot r & -r \\
    0 & 0 & 1 & q & r & 0 \\
    0 & 0 & 0 & 1 & 0 & 0 \\
    0 & 0 & 0 & -a & 1 & 0 \\
    0 & 0 & 0 & ac-b & -c & 1
\end{pmatrix}. \end{align*} The matrix $A$ can be written as, 
\begin{align*}
 A&=\begin{pmatrix}
    \gamma & O \\
    O & (\gamma^T)^{-1}
\end{pmatrix}\cdot \begin{pmatrix}
    I_n &\gamma^{-1}\cdot \delta \\
    O &I_n
\end{pmatrix} \\[8pt] &=\begin{pmatrix}
    1&a&b&0&0&0\\
    0&1&c&0&0&0\\
    0&0&1&0&0&0\\
    0&0&0&1&0&0\\
    0&0&0&-a&1&0\\
    0&0&0&ac-b&-c&1
\end{pmatrix} \cdot \begin{pmatrix}
    1&0&0&0&cq-p&-q\\
    0&1&0&-cq+p&0&-r\\
    0&0&1&q&r&0\\
    0&0&0&1&0&0\\
    0&0&0&0&1&0\\
    0&0&0&0&0&1
\end{pmatrix}\\[5pt] 
&={oe_{13}(b)\cdot oe_{23}(c)\cdot oe_{12}(a)\cdot oe_{15}(cq-p)\cdot oe_{16}(-q)\cdot oe_{26}(-r)}.
\end{align*}
\begin{remark}
    ${\rm TO}_{6}(R)$ is generated by the matrices $oe_{12}(z)$, $oe_{13}(z)$, $oe_{15}(z)$, $oe_{16}(z)$, $oe_{23}(z)$, and $oe_{26}(z)$ with $z\in R.$
\end{remark}
\subsection{Generators for ${\rm TO}_{2n+1}(R)$}
 Consider the following matrices obtained by the six generators of ${\rm TO}_{6}(R)$:
\begin{enumerate}[label=\roman*)]
    \item $1\perp oe_{12}(z)=F_{12}^3(z)=[F_1^1(z),F_2^2(-\frac{1}{2})]$,
    \item $1\perp oe_{13}(z)=F_{13}^3(z)=[F_1^1(z),F_3^2(-\frac{1}{2})]$,
    \item $1\perp oe_{15}(z)=F_{12}^4(z)=[F_2^1(z),F_1^1(\frac{1}{2})]$,
    \item $1\perp oe_{16}(z)=F_{13}^4(z)=[F_3^1(z),F_1^1(\frac{1}{2})]$,
    \item $1\perp oe_{23}(z)=F_{23}^3(z)=[F_2^1(z),F_3^2(-\frac{1}{2})]$,
    \item $1\perp oe_{26}(z)=F_{23}^4(z)=[F_3^1(z),F_2^1(\frac{1}{2})].$
\end{enumerate}
From the commutator relations derived above, it is clear that the following matrices are sufficient to generate ${\rm TO}_7(R)$,
\begin{enumerate}[label=\roman*)]
    \item $H_1=F_1^1(z)=I_7+e_{2,1}(2z)-e_{1,5}(z)-e_{2,5}(z^2)$,
    \item $H_2=F_2^1(z)=I_7+e_{3,1}(2z)-e_{1,6}(z)-e_{3,6}(z^2)$,
    \item $H_3=F_3^1(z)=I_7+e_{2,1}(2z)-e_{1,5}(z)-e_{2,5}(z^2)$,
    \item  $H_4=F_1^2(-\frac{1}{2})=I_7-e_{5,1}(1)+e_{1,2}(\frac{1}{2})-e_{5,2}(\frac{1}{4}),
$
    \item $H_5=F_2^2(-\frac{1}{2})=I_7-e_{6,1}(1)+e_{1,3}(\frac{1}{2})-e_{6,3}(\frac{1}{4}),$
    \item $H_6= F_3^2(-\frac{1}{2})=I_7-e_{7,1}(1)+e_{1,4}(\frac{1}{2})-e_{7,4}(\frac{1}{4}).$
\end{enumerate}

\begin{definition}
    ${\rm TO}_7(R)$ is the subgroup of ${\rm EO}_7(R)$ generated by $H_i=F_i^1(z)$, $1\leq i\leq 3$, and $H_j=F_{j-3}^2(\frac{1}{2})$, $4\leq j\leq 6$.
\end{definition}
\noindent Generalize this to define ${\rm TO}_{2n+1}(R)$, as follows:
\begin{definition}
   For $n\geq 3$, ${\rm TO}_{2n+1}(R)$ is the subgroup of ${\rm EO}_{2n+1}(R)$ generated by the matrices $H_i=F_i^1(z)$, $1\leq i\leq n$, and $H_j=F_{j-n}^2(\frac{1}{2})$, $n+1\leq j\leq 2n$.
\end{definition}
\begin{remark}
    The matrix $A=\begin{pmatrix}
    1&O&O\\O&\gamma & \delta \\
    O &O& (\gamma^T)^{-1}
\end{pmatrix}\in {\rm TO}_{2n+1}(R)$, where $\gamma$ is an upper triangular matrix such that $\gamma^{-1}\delta \in {{\rm Alt}}_n(R)$, and with identity elements on the main diagonal.
\end{remark}
\subsection{Generators for ${\rm TO'}_{2n+1}(R)$}
We can define ${\rm TO'}_{2n+1}(R)$ as follows: The general form of an element in ${\rm TO'}_{2n}(R)$ is given by 
$\begin{pmatrix}
    \gamma & O\\ \delta & (\gamma^T)^{-1}
\end{pmatrix}$, 
where $\gamma$ is a lower triangular matrix  with 1's on the main diagonal such that $\delta\cdot \gamma^{-1}\in {{\rm Alt}}_n(R)$. An arbitrary element in ${\rm TO'}_6(R)$ is of the form,

\begin{align*}
A&=\begin{pmatrix}
    I_n & O \\
    \delta\cdot\gamma^{-1} &I_n
\end{pmatrix}\cdot\begin{pmatrix}
    \gamma & O \\
    O & (\gamma^T)^{-1}
\end{pmatrix}\\[8pt]
&= \begin{pmatrix}
    1&0&0&0&0&0&0\\
    0&1&0&0&0&0\\
    0&0&1&0&0&0\\
    0&-cq+p&q&1&0&0\\
    cq-p&0&r&0&1&0\\
    -q&-r&0&0&0&1
\end{pmatrix} \cdot\begin{pmatrix}
    1&0&0&0&0&0\\
    a&1&0&0&0&0\\
    b&c&1&0&0&0\\
    0&0&0&1&-a&ac-b\\
    0&0&0&0&1&-c\\
    0&0&0&0&0&1
\end{pmatrix}\\[5pt]
&=oe_{51}(cq-p)\cdot oe_{61}(-q)\cdot oe_{62}(-r)\cdot oe_{31}(b)\cdot oe_{32}(c)\cdot oe_{21}(a)
\end{align*} 
Thus, we can define ${\rm TO'}_{2n+1}(R)$ as follows:

\begin{definition}
For $n\geq 3$, ${\rm TO'}_{2n+1}(R)$ is the subgroup of ${\rm EO}_{2n+1}(R)$ generated by the matrices $H_i= F_i^2(z)$, $1\leq i\leq n-1$, $H_i=F_{i-n}^1(z)$, $n+2\leq i\leq 2n$, and $F_n^2(\frac{1}{2})$.
\end{definition}
\begin{remark}
    The matrix $\begin{pmatrix}
   1&O&O\\ O& \gamma & O \\
    O&\delta & (\gamma^T)^{-1}
\end{pmatrix}\in {\rm TO'}_{2n+1}(R)$, where $\gamma$ is a lower triangular matrix with 1's on the main diagonal such that $\delta \gamma^{-1}\in {{\rm Alt}}_n(R)$.
\end{remark}

\section{Technical Results}

 We begin the section by giving results that identify odd orthogonal matrices and odd elementary orthogonal matrices. Then, we provide some results for transvections, which is vital in the upcoming sections. The main result we prove in this section is Theorem~\ref{Aux4.15}, which gives complete information about the canonical mapping $\phi$ from the subgroups ${\rm TO}_{2n+1}(R)$, ${\rm \Pi O}_{2n+1}(R)$, ${\rm MO}_{2n+1}(R)$, and ${\rm DO}_{2n+1}(R)$ on to ${\rm TO}_{2n+1}(R/I)$, ${\rm \Pi O}_{2n+1}(R/I)$, ${\rm MO}_{2n+1}(R/I)$, and ${\rm DO}_{2n+1}(R/I)$, where $R$ is a ring and $I$ is an ideal of $R$. Finally, the splitting of an odd dimensional orthogonal matrix with entries from a field $k$, is given by $ {\rm O}_{2n+1}(k)={\rm TO}_{2n+1}(k)\cdot {\rm MO}_{2n+1}(k)\cdot {\rm TO}_{2n+1}(k)$.
\begin{theorem}\label{Aux4.1}
 \begin{enumerate}[label=\emph{(\roman*)}]
    \item If $\alpha \in {\rm GL}_n(I)$, then $\begin{pmatrix}
        1&O&O\\O&\alpha&O\\O&O&(\alpha^T)^{-1}
    \end{pmatrix}\in {\rm O}_{2n+1}(I)$.
    \item If $\alpha \in {\rm E}_{n}(R,I)$, then $\begin{pmatrix}
        1&O&O\\O&\alpha&O\\O&O&(\alpha^T)^{-1}
    \end{pmatrix}\in {\rm EO}_{2n+1}(R,I).$
    \end{enumerate}
\end{theorem}
\begin{proof}
    The first assertion simply follows from the definition of an odd orthogonal group. Assertion (ii) follows from Lemma~2.2 of \cite{SusKop1977}.
\end{proof}
\begin{theorem}\label{Aux4.2}
    If $\alpha \in {{\rm Alt}}_n(I)$, then the matrices $\begin{pmatrix}
        1&O&O\\O&I_n&\alpha\\O&O&I_n
    \end{pmatrix}$ and $\begin{pmatrix}
        1&O&O\\O&I_n&O\\O&\alpha&I_n\end{pmatrix}$ belongs to ${\rm EO}_{2n+1}(I)$.
\end{theorem}
\begin{proof}
    Note that,\begin{align*}
        \begin{pmatrix}
        1&O&O\\O&I_n&\alpha\\O&O&I_n
    \end{pmatrix}&=\begin{pmatrix}
        1&O&O\\
        O&I_n&\displaystyle \sum_{1\leq i,j\leq n} \alpha_{ij}(e_{ij}-e_{ji})\\
        O&O&I_n
    \end{pmatrix}\\&=\displaystyle\prod_{1\leq i,j\leq n}\begin{pmatrix}
        1&O&O\\O&I_n&\alpha_{ij}(e_{ij}-e_{ji})\\O&O&I_n
    \end{pmatrix}\\&=\displaystyle\prod_{1\leq i,j\leq n}{\rm {1}} \perp {oe}_{i+1,n+j+1}(\alpha_{ij})=\displaystyle\prod_{1\leq i,j\leq n}F_{i,j}^4 (\alpha_{ij}).
    \end{align*}
    By similar arguments, we obtain $$
        \begin{pmatrix}
        1&O&O\\O&I_n&O\\O&\alpha&I_n
    \end{pmatrix}=\displaystyle\prod_{1\leq i,j\leq n}F_{i,j}^5 (\alpha_{ij}).$$
\end{proof}
\noindent The following results in \cite{SusKop1977} are some useful tools for the main results in this section.
\begin{lemma}\label{Aux4.3}{\rm \cite[Lemma~2.5]{SusKop1977}}
    Suppose $v,w\in R^{n}$ such that $v^T\cdot w=0.$ If $x\in o(w)$, then there exists $\alpha \in { {\rm Alt}}_{n}(R)$, such that $v\cdot x=\alpha\cdot w.$
\end{lemma}
\begin{corollary}\label{Aux4.4}{\rm \cite[Corollary~2.6]{SusKop1977}}
    Assume that $n\geq 3$, $v,w\in R^{n}$ and $v^T\cdot w=0$. If $x\in o(v)+o(w)$, then $I_{n+1}+x\cdot(v\cdot w^T)\in {\rm E}_{n}(R)$.
\end{corollary}
\begin{corollary}\label{Aux4.5}{\rm \cite[Corollary~2.7]{SusKop1977}}
    Assume that $n\geq 3$, $v',v'',w'\in R^{n}$ and $(v'')^T\cdot v'=(v'')^T\cdot w'=0$. If $x\in(o(v')+o(v''))^2$, then $I_n+x\cdot(v''\cdot (w')^T)\in {\rm E}_{n}(R)$.
\end{corollary}
\subsection{Matrix form of the transvection ${\rm E}_{v,w}^o(x)$}

Given $v=\begin{pmatrix}
    v_0&v'&v''
\end{pmatrix}^T\in R^{2n+1}$ and $w=\begin{pmatrix}
    w_0&w'&w''
\end{pmatrix}^T \in R^{2n+1}$,
we denote $\tilde{v}=\begin{pmatrix}
    2v_0&(v'')^T&(v')^T
\end{pmatrix}$ and  $\tilde{w}=\begin{pmatrix}
    2w_0&(w'')^T&(w')^T
\end{pmatrix}$. By this notation, bilinear form $\phi$ can be written as $\phi(v,w)=\tilde{v}\cdot w=\tilde{w}\cdot v$. If $q(v)=\phi(v,w)=0$, the transvection ${\rm E}^o_{v,w}(x)$ is defined in matrix form as follows, 
\begin{equation}
{\rm E}_{v,w}^o(x)=I_{2n+1}+x\cdot(w\cdot\tilde{v}-v\cdot\tilde{w})-x^2\cdot q(w)\cdot(v\cdot \tilde{w}).
\end{equation}
If $q(w)=0$, this expression reduces to ${\rm E}_{v,w}^o(x)=I_{2n+1}+x\cdot(w\cdot\tilde{v}-v\cdot\tilde{w})$.
\begin{lemma}\label{Aux4.6}
    For $n\geq3$, let $v=\begin{pmatrix}
        v_0&v_1&\dots&v_{2n}
    \end{pmatrix}^T$, $w=\begin{pmatrix}
        w_0&w_1& \dots &w_{2n}
    \end{pmatrix}^T\in R^{2n+1}$, with $v_0^2=w_0^2=v_0\cdot w_0=0$ and assume that $q(v)$, $\phi(v,w)$, $q(w)=0$ and $w_i=0$ or $w_{\delta(i)}=0$, where $i=1,\ldots ,n$. If $x\in o(v)^2$, then ${\rm E}_{v,w}^o(x)\in {\rm EO}_{2n+1}(R)$.
\end{lemma}
\begin{proof}
    Let $\sigma_\pi$ be the matrix corresponding to the permutation $\pi$ which commutes with $\delta$. Then $\sigma_\pi\in {\rm O}_{2n+1}(R)$ and it normalizes ${\rm EO}_{2n+1}(R).$ By hypothesis, we get a permutation $\pi$ commuting with $\delta$ such that $\sigma_\pi \cdot w$ has the form $\begin{pmatrix}
        w_0&w'&O
    \end{pmatrix}^T$, where $w'\in R^n$. Assume that $w=\begin{pmatrix}
        w_0&w'&O
    \end{pmatrix}^T $and write $v=\begin{pmatrix}
        v_0&v'&v''
    \end{pmatrix}^T$. Then by hypothesis, $(v'')^T\cdot w'=(v'')^T\cdot v'=(w')^T\cdot v''=(v')^T\cdot v''=0$ and the matrix corresponding to ${\rm E}_{v,w}^o(x)$ is,  
    \begin{align*}        
    {\rm E}_{v,w}^o(x)&=I_{2n+1}+x\cdot(w\cdot\tilde{v}-v\cdot\tilde{w})\\[8pt] &=\begin{pmatrix}
        1&x\cdot (w_0 \cdot(v'')^T )&x\cdot( w_0\cdot(v')^T-v_0\cdot(w')^T)\\
        2x\cdot(w' \cdot v_0-v'\cdot w_0)& I_n-x\cdot(w'\cdot (v'')^T)& x\cdot(w'\cdot(v')^T-v'\cdot (w')^T)\\
        -2x\cdot(v''\cdot w_0)&0&I_n+x\cdot(v''\cdot(w')^T)
    \end{pmatrix}\\[8pt] 
    &=\begin{pmatrix}
        1&O&O\\O&\alpha&O\\O&O&(\alpha^T)^{-1}
    \end{pmatrix} \begin{pmatrix}
        1&O&O\\O&I_n&x(v'\cdot(w')^T-w' \cdot(v')^T)\\
        O&O&I_n
    \end{pmatrix} \begin{pmatrix}
       1&\beta_1 & \beta_2\\
       -2(\beta_2)^T&I_n&O\\
       -2(\beta_1)^T&O&I_{n}
    \end{pmatrix}.
    \end{align*}
Here $\alpha=I_n-x\cdot(w' \cdot(v'')^T)$, $\beta_1=x\cdot w_0 \cdot(v'')^T $ and $\beta_2=x\cdot( w_0\cdot(v')^T-v_0\cdot(w')^T)$. Using Corollary~\ref{Aux4.4}, the matrix $(\alpha^T)^{-1}=I_n+x\cdot(v''\cdot (w')^T)\in {\rm E}_n(R)$. Hence $\alpha\in {\rm E}_n(R)$ and the first factor is found in ${\rm EO}_{2n+1}(R)$ by Theorem~\ref{Aux4.1}. Similarly, using Theorem~\ref{Aux4.2}, the second factor lies in ${\rm EO}_{2n+1}(R)$. Also, the third factor can be written in the form (since we have the conditions $v_0^2=w_0^2=v_0\cdot w_0=0$), $$\begin{pmatrix}
        1&x\cdot w_0 \cdot(v'')^T & O\\
        O&I_n&O\\
        -2x\cdot v'' w_0&-(x\cdot w_0 \cdot(v'')^T)^2&I_{n}
    \end{pmatrix}\cdot \begin{pmatrix}
        1&O & x\cdot( w_0\cdot(v')^T-v_0\cdot(w')^T)\\
        2x\cdot(w'\cdot v_0-v'\cdot w_0)&I_n&-(x\cdot( w_0\cdot(v')^T-v_0\cdot(w')^T)^2\\
        O&O&I_{n}
    \end{pmatrix},$$ hence it lies in ${\rm EO}_{2n+1}(R).$
\end{proof}
\begin{lemma}\label{Aux4.7}
    Let $n\geq 3$, $\alpha\in {\rm Alt}_{n+1}(R)$ and $y\in R.$ Assume  $v=\begin{pmatrix}
        v_0&v'&v''
    \end{pmatrix}^T\in R^{2n+1}$, with $v_0^2=0$ and $ u=\begin{pmatrix}
             2u_0\\u'
        \end{pmatrix}\in R^{n+1}$, and also $q(v)=0$ and $ \begin{pmatrix}
            v_0\\v'
        \end{pmatrix}\cdot y=\alpha\cdot \begin{pmatrix}
            v_0\\v''
        \end{pmatrix} $. Put $w=\begin{pmatrix}
        \alpha\cdot u\\u'\cdot y
    \end{pmatrix}$ and assume $(u_0\cdot y)=\alpha_0\cdot u$, where $\alpha_0$ is the $1^{\rm st}$ row of $\alpha$. Then we have the following.
    \begin{enumerate}[label=\emph{(\roman*)}]
        \item $q(w)=\phi(v,w)=0$.
        \item If $x\in o(v)^2$, then $ {\rm E}_{v,w}^o(x)\in {\rm EO}_{2n+1}(R)$.
    \end{enumerate}
\end{lemma}
\begin{proof}
   \begin{enumerate}[label=(\roman*)]
       \item Let $\alpha=\begin{pmatrix}
           \alpha_{0}\\\alpha_{1}\\\alpha_{2}\\ \vdots\\\alpha_{n+1}
       \end{pmatrix}=\begin{pmatrix}
           0&\alpha_{12}&\alpha_{13}&\dots&\alpha_{1n+1}\\
           -\alpha_{12}&0&\alpha_{23}&\dots&\alpha_{2n+1}\\
           -\alpha_{13}&-\alpha_{23}&0&\dots&\alpha_{3n+1}\\
           \vdots&\vdots&\vdots&\ddots&\vdots\\
           -\alpha_{1n+1}&-\alpha_{2n+1}&-\alpha_{3n+1}&\dots&0
       \end{pmatrix}$, where $\alpha_{i}$ denote the\\[6pt] $(i+1)^{\rm th}$ row, for $i=0,\dots ,n$. Then we get, $$w=\begin{pmatrix}
       \alpha_{12}u_1'+\alpha_{13}u_2'+\alpha_{14}u_3'+\dots+\alpha_{1n+1}u_n'\\[4pt]
-2\alpha_{12}u_0+\alpha_{23}u_2'+\alpha_{24}u_3'+\dots+\alpha_{2n+1}u_n'\\[4pt]
-2\alpha_{13}u_0-\alpha_{23}u_1'+\alpha_{34}u_3'+\dots+\alpha_{3n+1}u_n'\\[4pt]
\vdots\\[4pt]
-2\alpha_{1n+1}u_0-\alpha_{2n+1}u_1'-\alpha_{3n+1}u_2'-\dots -\alpha_{nn+1}u_{n-1}'\\[4pt]u_1'y\\[4pt]u_2'y\\[4pt] \vdots\\[4pt]u_n'y
       \end{pmatrix}.$$ Since we have the condition $\alpha_{12}u_1'+\alpha_{13}u_2'+\alpha_{14}u_3'+\dots+\alpha_{1n+1}u_n'=(u_0\cdot y)$, by computations we get $q(w)=0$. We get \begin{align*}
\phi(v,w)&=2v_0\alpha_0u+v_1''\alpha_1u+\dots+v_n''\alpha_nu+v_1'u_1'y+\dots+v_n'u_n'y\\[6pt]&=\begin{pmatrix}
    v_0&(v'')^T
\end{pmatrix}(\alpha\cdot u)+\begin{pmatrix}
     v_0&(v')^T
\end{pmatrix}(u\cdot y)\\[6pt]&=-(\alpha\cdot\begin{pmatrix}
    v_0\\v''
\end{pmatrix})^T\cdot u+(\begin{pmatrix}
            v_0\\v'
        \end{pmatrix}\cdot y)^T\cdot u=0\end{align*}
       
   \item Since, we have the mapping $u\rightarrow {\rm E}_{v,w}^o(x)$ in Lemma~\ref{Pre2.3} (iii), which is a homomorphism of the additive group of $R^{n+1}$ into ${\rm O}_{2n+1}(R)$, we only need to consider the assertion for columns $u$ having only one non-zero entry. Then, the assertion is straightforward from Lemma~\ref{Aux4.6}.
 \end{enumerate}
\end{proof}
\begin{theorem}\label{Aux4.8}
Let $n\geq 3$, $v=\begin{pmatrix}
    v_0&v'&v''
\end{pmatrix}^T\in R^{2n+1}$, $w=\begin{pmatrix}
    w_0&w'&w''
\end{pmatrix}^T\in R^{2n+1}$ with $v_0^2=v_0\cdot w_0=0$, and $q(v)$, $\phi(v,w)=0$. If $x\in o(v)^3$, then the transvection, ${\rm E}_{v,w}^o(x)$ lies in ${\rm EO}_{2n+1}(R)$.
\end{theorem}
\begin{proof}
    Consider the mapping $x\mapsto {\rm E}_{v,w}^o(x)$, which is a homomorphism of the additive group of the ring $R$ into ${\rm O}_{2n+1}(R)$ (Lemma~\ref{Pre2.3}). It suffices to consider $x$, ranging over some system of generators of the additive group of $o(v)^3.$ Therefore, assume that $x=x_1\cdot y$, where $x_1\in o(v)^2$ and $y$ belongs\\[5pt] to either $o(\begin{pmatrix}
         \frac{1}{2}v_0\\v'
    \end{pmatrix})$ or $o(\begin{pmatrix}
         \frac{1}{2}v_0\\v''
    \end{pmatrix})$. Without loss of generality, assume that $y\in o(\begin{pmatrix}
        \frac{1}{2}v_0\\v''
    \end{pmatrix})$. Since $q(v)=0$, we get $\frac{1}{4}v_0^2+(v')^T\cdot v''=\begin{pmatrix}
        \frac{1}{2}v_0\\v'
    \end{pmatrix}^T\cdot \begin{pmatrix}
         \frac{1}{2}v_0\\v''
    \end{pmatrix}=0$. Using Lemma~\ref{Aux4.3}, there exists $\alpha\in {\rm Alt}_{n+1}(R)$ such that\\[5pt]$\begin{pmatrix}
         \frac{1}{2}v_0\\v'
    \end{pmatrix}\cdot y=\alpha\cdot\begin{pmatrix}
        \frac{1}{2}v_0\\
    v''\end{pmatrix} $. Set $w_1=\begin{pmatrix}
        \alpha\cdot \begin{pmatrix} w_0\\ w''\end{pmatrix}\\w''\cdot y
    \end{pmatrix}$ and  $w_2=\begin{pmatrix}
        \begin{pmatrix} w_0\\ w'\end{pmatrix}\cdot y-\alpha\cdot \begin{pmatrix} w_0\\ w''\end{pmatrix}\\0
    \end{pmatrix}$, with $(w_0\cdot y)=\alpha_0\cdot w$, clearly $w_1+w_2=w\cdot y$. Using Lemma~\ref{Pre2.3}, we have ${\rm E}_{v,w}^o(x)={\rm E}_{v,w\cdot y}(x_1)={\rm E}_{v,w_1}^o(x_1)\cdot {\rm E}_{v,w_2}^o(x_1)$. Here, assume $(w_0\cdot y)=\alpha_0\cdot \begin{pmatrix}
        w_0\\w''
    \end{pmatrix}$, where $\alpha_0$ is the $1^{\rm st}$ row of $\alpha$, then both the factors lies in ${\rm EO}_{2n+1}(R)$ by Lemma~\ref{Aux4.7} and Lemma~\ref{Aux4.6} respectively.
\end{proof}
\begin{corollary} \label{Aux4.9}
    Assume that $n\geq 3$, $v,w\in R^{2n+1}$ with $v_0^2=w_0^2=v_0\cdot w_0=0$, and $q(v)=q(w)=\phi(v,w)=0$. If $x=o(v)^3+o(w)^3$, then ${\rm E}_{v,w}^o(x)\in {\rm EO}_{2n+1}(R)$.
\end{corollary}
The normality theorem for the DSER elementary orthogonal group ${\rm EO}_R(Q,{\mathbb H}(R)^m)$ proved by A.A. Ambily in \cite{AmbRao2020} leads us to the following result: 
 \begin{theorem}\label{Aux4.10} For $n\geq 3$, the odd elementary orthogonal group ${\rm EO}_{2n+1}(R)$ is a normal subgroup of ${\rm O}_{2n+1}(R)$.
\end{theorem}
\begin{lemma}\label{Aux4.11}
Let $A\subset C\subset B$ be a tower of rings and  $A$ is a retract in  $B$. Then 
\begin{enumerate}[label=\emph{(\roman*)}]
    \item ${\rm O}_{2n+1}(A) \cap {\rm EO}_{2n+1}(B) = {\rm EO}_{2n+1}(A)$.
    \item $({\rm O}_{2n+1}(A) \cdot {\rm EO}_{2n+1}(C)) \cap {\rm EO}_{2n+1}(B) = {\rm EO}_{2n+1}(C).$
\end{enumerate}
\end{lemma}

\begin{proof}
\begin{enumerate}[label=(\roman*)]
\item  Clearly ${\rm EO}_{2n+1}(A)\subset {\rm O}_{2n+1}(A) \cap {\rm EO}_{2n+1}(B).$  Let $\alpha\in {\rm O}_{2n+1}(A) \cap {\rm EO}_{2n+1}(B)$, then $\alpha=\psi(\alpha)$, where $\psi$ is the retraction. Thus, it belongs to ${\rm EO}_{2n+1}(A).$ 
\item It is obvious that, ${\rm EO}_{2n+1}(C) \subset ({\rm O}_{2n+1}(A) \cdot {\rm EO}_{2n+1}(C)) \cap {\rm EO}_{2n+1}(B)$. For the reverse inclusion, let $\alpha\in ({\rm O}_{2n+1}(A) \cdot {\rm EO}_{2n+1}(C)) \cap {\rm EO}_{2n+1}(B)$, then $\alpha=\alpha_1\cdot\alpha_2$, where $\alpha_1\in {\rm O}_{2n+1}(A)$ and $\alpha_2\in {\rm EO}_{2n+1}(C) $. Then $\alpha_1=\alpha\cdot{\alpha_2}^{-1}\in {\rm EO}_{2n+1}(B)\cap {\rm O}_{2n+1}(A)={\rm EO}_{2n+1}(A)$, so that $\alpha\in {\rm EO}_{2n+1}(C).$
  \end{enumerate}
\end{proof}

As a special case of \cite[Lemma~2.2]{Rao1984}, we obtain the following theorem. 
\begin{theorem}\label{Aux4.12}
  Let $C=k[X]$, $k[X^{-1}]$, or $k[X,X^{-1}]$, where $k$ is a field. For $n\geq2$, ${\rm O}_{2n+1}(C)={\rm O}_2(C)\cdot {\rm EO}_{2n+1}(C)$.
\end{theorem}
\begin{corollary}\label{Aux4.13}
    For $n\geq2$, the following equality holds, ${\rm O}_{2n+1}(C)\cap {\rm EO}(C)= {\rm EO}_{2n+1}(C).$
\end{corollary}
\begin{proof}
    By Theorem~\ref{Aux4.12}, it is sufficient to show that ${\rm O}_2(C)\cap {\rm EO}(C)\subset {\rm EO_5}(C).$
    Let $\alpha\in{\rm O}_2(C)$, then we have $$\alpha=\begin{pmatrix}
        c&0\\0&c^{-1}
    \end{pmatrix} ~~{~\rm or} ~~~~\alpha=\begin{pmatrix}
        0&c\\c^{-1}&0
    \end{pmatrix},~{\rm where}~ c\in {\rm GL}_1(C).$$
    Also, since $\alpha\in {\rm EO}(C) $, the image of $\alpha$ in ${\rm KO_1}(F)=\mathbb{Z}/2\mathbb{Z}\times \dot F/\dot F^2$ should be zero, where $F$ is the field of quotients of $C$ and $\dot F$ is it's multiplicative group \cite[Section 5.3.5]{HahMer1989}. Thus, $\alpha$ must be of the form, $$\alpha=\begin{pmatrix}
        c&0\\0&c^{-1}
    \end{pmatrix}, ~~ {\rm where}~ c\in \dot F^2 ~{\rm and}~ c=b^2~ {\rm for ~some}~ b\in {\rm GL}_1(C).$$
  Let $\sigma_\pi$  be the permutation matrix corresponding to the permutation $\pi=(2,4)$, then we notice the following equality,
    $$\begin{pmatrix}
        1&0&0&0&0\\
        0&b^2&0&0&0\\
        0&0&1&0&0\\
        0&0&0&(b^{-1})^2&0\\
        0&0&0&0&1
    \end{pmatrix}=\begin{pmatrix}
        1&0&0&0&0\\
        0&b&0&0&0\\
        0&0&b^{-1}&0&0\\
        0&0&0&b^{-1}&0\\
        0&0&0&0&b
    \end{pmatrix}\cdot \sigma_\pi\cdot\begin{pmatrix}
        1&0&0&0&0\\
        0&b&0&0&0\\
        0&0&b^{-1}&0&0\\
        0&0&0&b^{-1}&0\\
        0&0&0&0&b
    \end{pmatrix}\cdot (\sigma_\pi)^{-1}.$$
    Hence, our assertion follows from Theorem~\ref{Aux4.1} and the fact that $diag(b,b^{-1})\in {\rm E}_2(C)$.
   
\end{proof}

The following result can be obtained by using Local-Global Principle \cite[Theorem~5.1]{AmbRao2014}.
\begin{theorem}\label{Aux4.14}
    Let $\beta \in {\rm O}_{2n+1}(R[X])$, where $n\geq3$. If  $\beta_{\mathfrak{m}} \in {\rm O}_{2n+1}(R_{\mathfrak{m}})\cdot {\rm EO}_{2n+1}(R_{\mathfrak{m}}[x])$, for each ${\mathfrak{m}}\in {\rm Max}(R)$, then $\beta \in {\rm O}_{2n+1}(R)\cdot {\rm EO}_{2n+1}(R[X])$.
    \end{theorem}

    \begin{theorem}\label{Aux4.15}
        Let $I$ be an ideal of the ring $R$. Let  $A$ denote the quotient ring $R/I$ and $\phi : R\rightarrow A$ be the canonical homomorphism. Then,
   \begin{enumerate}[label=\emph{(\roman*)}]
            \item For any $\alpha \in {\rm TO}_{2n+1}(A)$ $($or $\rm {\Pi O}_{2n+1}(A)$ $)$, there exists $\beta \in {\rm TO}_{2n+1}(R)$ $($or ${\rm {\Pi O}_{2n+1}}(R)$$)$ such that $\phi(\beta)=\alpha.$
            \item If the induced homomorphism ${\rm GL}_1(R)\rightarrow {\rm GL}_1(A)$ is surjective, then for any $\alpha \in {\rm DO}_{2n+1}(A)$ $($or ${\rm {MO}_{2n+1}}(R)$$)$, there exists $\beta \in {\rm DO}_{2n+1}(R)$ $($or ${\rm {MO}_{2n+1}}(R)$$)$ such that $\phi(\beta)=\alpha.$
            \item If $\beta \in {\rm TO}_{2n+1}(R)$ and 
            $\phi(\beta)=I_{2n+1}$, then $\beta \in {\rm EO}_{2n+1}(R,I).$ 
        \end{enumerate}
    \end{theorem}
    \begin{proof}
        Assertion (i) follows, since the map $\phi$ is surjective.  If the mapping ${\rm{GL}_1}(R)\rightarrow {\rm{GL}_1}(A)$ is surjective and $\alpha = diag(d_0,d_1,\dots,d_n,d_1^{-1},\dots,d_n^{-1})$, where $d_i\in {\rm{GL}_1}(A)$, for $i=0,1,\dots,n$, we can find $b_i\in {\rm{GL}_1}(R)$ such that $\phi(b_i)=d_i$ and $\phi(b_0)=d_0$, and put $\beta=diag(b_0,b_1,\dots,b_n,b_1^{-1},\dots,b_n^{-1}).$
        Thus, the Part  of (ii) pertaining to the group ${\rm DO}_{2n+1}(A)$ is proved. To prove the result for the subgroup ${\rm MO}_{2n+1}(A)$, we can use the relation, ${\rm {MO}}_{{2n+1}}(R)={\rm \Pi O}_{{2n+1}}(R)\cdot {\rm {DO}}_{{2n+1}}(R)={\rm {DO}}_{{2n+1}}(R) \cdot {\rm {\Pi O}}_{{2n+1}}(R)$ and Part (i).

We can prove (iii) for the generators of ${\rm  TO}_{2n+1}(R)$. Let $\beta=F_i^1(z)$. Then we have $\phi(\beta)=I_{2n+1}$. That is, $\phi(I_{2n+1}+e_{1,n+i+1}(z)-e_{i+1,1}(2z)-e_{i+1,n+i+1}(z^2))=I_{2n+1}$. Hence, $\phi(e_{1,n+i+1}(z))=\phi(e_{i+1,1}(2z))=\phi(e_{i+1,n+i+1}(z^2))=0$, which implies that $\beta\in {\rm EO}_{2n+1}(R,I)$. Let $\beta=F_{i}^2(\frac{1}{2})=I_{2n+1}+e_{1,i+1}(\frac{1}{2})-e_{n+i+1,1}(1)-e_{n+i+1,i+1}(\frac{1}{4})$, then clearly $\beta$ will not map to the identity.
Hence, we have the result.    
\end{proof}
   
   Using Theorem~\ref{Aux4.12}, we can decompose ${\rm O}_7(k)$ as ${\rm O}_2(k)\cdot {\rm EO}_7(k)$, for a field $k$. Thus, we can identify the generators of ${\rm O}_7(k)$ by $G_i={\rm O}_2(k)\cdot F_i^1(z)$, for $1\leq i\leq 3$ and $G_j={\rm O}_2(k)\cdot F_{j-3}^2(z)$, for $4\leq j \leq 6$.

\begin{lemma}\label{Aux4.16}
For a field $k$, $ {\rm O}_7(k)={\rm TO}_7(k)\cdot {\rm MO}_7(k)\cdot {\rm TO}_7(k)$.
\end{lemma}
\begin{proof}
It is sufficient to prove that for each $G_i$, we can get some $\tau_{1i}, \tau_{2i}\in {\rm TO}_7(k)$ and $\mu_i\in {\rm MO}_7(k)$ such that $\tau_{1i}\cdot G_i\cdot \tau_{2i}=\mu_i$. For $G_i$ ($1\leq i\leq 3$), choose $\tau_{1i}=I_7$ and $\tau_{2i}=F_i^1(-z)$. Then, we get $\mu_i$ as $[c]\perp[c^{-1}]$ or $[c]\top[c^{-1}]$ for a non-zero $c \in k$. We need to obtain the combinations, $\tau_{14}\cdot G_4 \cdot \tau_{24}\in {\rm MO}_7(k)$, 
         $\tau_{15}\cdot G_5\cdot \tau_{25}\in {\rm MO}_7(k)$, and
         $\tau_{16}\cdot G_6 \cdot \tau_{26}\in {\rm MO}_7(k).$
   \vspace{2mm}

    If $z=0$, it is obvious. Let $z\neq 0$, then we obtain the following relations,\begin{align*}  
     F_1^2(z)&=([H_3,[H_4,H_6]]\cdot[F_3^1(-\frac{z^2}{2}),H_4]^{-1})^2\\[4pt]&=([F_3^1(z),[F_1^2(-\frac{1}{2}),F_3^2(-\frac{1}{2})]]\cdot[F_3^1(-\frac{z^2}{2}),F_1^2(-\frac{1}{2})])^2,\\[4pt]
    F_2^2(z)&=([H_3,[H_5,H_6]]\cdot[F_3^1(-\frac{z^2}{2}),H_5]^{-1})^2\\[4pt]&=([F_3^1(z),[F_2^2(-\frac{1}{2}),F_3^2(-\frac{1}{2})]]\cdot[F_3^1(-\frac{z^2}{2}),F_2^2(-\frac{1}{2})])^2,~~{\rm and}\\[4pt]   
    F_3^2(z)&=([F_2^1(\frac{z^2}{2}),H_6]\cdot[H_2,[H_5,H_6]]^{-1})^2\\[4pt]&=([F_2^1(\frac{z^2}{2}),[F_3^2(-\frac{1}{2})])^2 \cdot[F_2^1(z),F_2^2(-\frac{1}{2}),F_3^2(-\frac{1}{2})]).
    \end{align*} 
By taking $\tau_{1j}=I$ and $\tau_{2j}=F_i^2(-z)$, for $4\leq j\leq 6, $ we get the required matrix as $[c]\perp[c^{-1}]$ or $[c]\top[c^{-1}]$, which is in ${\rm MO}_7(k)$.
\end{proof}

\begin{theorem}\label{Aux4.17}
For a field $k$, we have ${\rm O}_{2n+1}(k)={\rm TO}_{2n+1}(k)\cdot {\rm MO}_{2n+1}(k)\cdot {\rm TO}_{2n+1}(k)$.
\end{theorem}
\begin{proof}
    We can prove the result by induction on $n$. For $n=3$, by Lemma~\ref{Aux4.16}, the result holds. Suppose $n>3$, and $\alpha\in \mathrm{O}_{2n+1}(k)$. 
     Let $v$ be the $2^{\rm nd}$ column of $\alpha$, which is given by  $v=\begin{pmatrix}
a&v^\prime&v^{\prime\prime}
    \end{pmatrix}^T$,  where $v^\prime=\begin{pmatrix}
        v_1^\prime& \dots&v_n^\prime
    \end{pmatrix}^T$ and $v^{\prime\prime}=\begin{pmatrix}
        v_1^{\prime\prime}&\dots&v_n^{\prime\prime}\end{pmatrix}^T$.
    Assuming $v^{\prime\prime}\neq 0$, there exists a lower triangular matrix $\gamma$ such that $\gamma \cdot v''$ has only one non-zero entry other than the $1^{\rm st}$ position.
     Since $(v')^T\cdot v''=q(v)=0$ and $o(v'')=k$, by Lemma~\ref{Aux4.3}, there exists $\delta \in {{\rm Alt}}_n(k)$ such that $v'=\delta\cdot v''$. The matrix
$$\sigma=\begin{pmatrix}
   1&O&O\\O& (\gamma^T)^{-1} & O \\
   O& O & \gamma
\end{pmatrix}\cdot \begin{pmatrix}
    1&O&O\\O&I_n & -\delta \\
   O& O & I_n
\end{pmatrix}
\text {lies in}\ \mathrm{TO}_{2n+1}(k).$$ 
Also, $\sigma\cdot v$ has only one non-zero entry other than the $1^{\rm st}$ position, 
$$\sigma\cdot v
 = \begin{pmatrix}
a\\(\gamma^T)^{-1}\cdot v'-(\gamma^T)^{-1}\cdot\delta \cdot v'' \\
\gamma\cdot v''
\end{pmatrix} 
= \begin{pmatrix}
a\\O \\
\gamma\cdot v''
\end{pmatrix}.$$

If $v''=0$, then $v'\neq0$, and there exists an upper triangular matrix $\gamma$ such that $\gamma\cdot v'$ has only one non-zero entry other than $a$. In this case, consider $$\sigma=\begin{pmatrix}
    1&O&O\\O&\gamma&O\\O&O&(\gamma^T)^{-1}
\end{pmatrix},$$ which lies in ${\rm TO}_{2n}(k)$ and $\sigma\cdot v$ has only one non-zero entry. Thus,
$$\sigma\cdot v 
= \begin{pmatrix}
a\\\gamma \cdot v' \\
(\gamma^T)^{-1} \cdot v''
\end{pmatrix} = \begin{pmatrix}
a\\\gamma \cdot v' \\O
\end{pmatrix}.$$

It suffices to show that $\sigma\cdot \alpha\in {\rm TO}_{2n+1}(k)\cdot \rm {\rm MO}_{2n+1}(k)\cdot {\rm TO}_{2n+1}(k)$ (Since ${\rm TO}_{2n+1}(k)$ is a subgroup). Assume that the $2^{\rm nd}$ column of $\alpha$ has only one non-zero entry $s$, other than the $1^{\rm st}$ position, say $(i+1)^{\rm th}$. 

\vspace{1mm}

 Let $u=(t,u',u'')$ denote the $(i+1)^{\rm th}$ row of $\alpha$. Clearly, $u_1'=s\neq0$. Hence, there exists an upper triangular matrix $\gamma$ with $u'\cdot \gamma=s\cdot e_1^T$ and an alternating matrix $\delta$ such that $u''=u'\cdot \delta$. Replacing $\alpha$ by, $$\alpha\cdot \begin{pmatrix}
    1&O&O\\O&I_n&-\delta\\
    O&O&I_n
\end{pmatrix}\cdot 
\begin{pmatrix}
    1&O&O\\O&\gamma&O\\ 
    O&O&(\gamma^T)^{-1}
\end{pmatrix},$$ 
we may also assume that all the entries of the $(i+1)^{\rm th}$ row of $\alpha$ except the $2^{\rm nd}$ entry are zero (except for $t$). Here, it is easy to see that $(n+2)^{\rm th}$ column of $\alpha$ is equal to $s^{-1}\cdot e_{\delta(i+1)}$, and the $(\delta(i+1))^{\rm th}$ row is equal to $s^{-1}\cdot e_{n+2}^T$. 
 Let $\alpha_1$ denote the matrix obtained from $\alpha$ by deleting the $2^{\rm nd}$ and $(n+2)^{\rm th}$ columns and $i^{\rm th}$ and $(\delta(i))^{\rm th}$ rows (that is, delete the terms involving $a$ and $a^{-1}$ in the construction). Clearly, $\alpha_1\in {\rm O}_{2n-1}(k)$, hence by induction hypothesis there exist $\tau_1,\tau_2\in {\rm TO}_{2n-1}(k)\ \text{such that} \ \tau_1\cdot\alpha_1\cdot\tau_2\in \rm MO_{2n-1}(k).$
 If we denote $\lambda_1$, the matrix all of whose elements in the $(i+1)^{\rm th}$ and $(\delta(i+1))^{\rm th}$ columns and also in the $(i+1)^{\rm th}$ and $(\delta(i+1))^{\rm th}$ rows are zero, except for the elements in positions $(i+1,i+1)$ and $(\delta(i+1),\delta(i+1))$, which are equal to the unity element, and after deleting from which the $(i+1)^{\rm th}$ and $(\delta(i+1))^{\rm th}$ rows and $(i+1)^{\rm th}$ and $(\delta(i+1))^{\rm th}$ columns, we obtain the matrix $\tau_1.$
\vspace{1mm}

 Similarly, define $\lambda_2$, as the canonical image of $\tau_2$ in ${\rm O}_{2n+1}(k)$, then it is clear that $\lambda_1,\lambda_2\in {\rm TO}_{2n+1}(k)$ and $\lambda_1\cdot\alpha\cdot\lambda_2\in{\rm MO}_{2n+1}(k).$
    \end{proof}
 
   \section{Elementary orthogonal group over a Laurent Ring}
   Consider the orthogonal group over a Laurent ring, $R[X, X^{-1}]$. Our aim is to decompose a matrix\\[5pt] in ${\rm EO}_{2n+1}(R[X, X^{-1}])\cap {\rm O}_{2n+1}(R[X,X^{-1}],{\mathfrak m}[X,X^{-1}])$ as a product of matrices from ${\rm EO}_{2n+1}(R[X],{\mathfrak m}[X])$\\[5pt] and ${\rm EO}_{2n+1}(R[X,X^{-1}],{\mathfrak m}[X,X^{-1}])$. \emph{This is a crucial step for proving Horrocks' theorem}. 
   \vspace{3mm}

        Let $\theta_n$ denote the matrix $diag(X,\dots X,1,\dots,1)$, where $X$ appears in the first $n$ diagonal positions.
\begin{lemma}\label{Ele5.1}
    For $n\geq3$, if $f\in R[X]$ and $\alpha\in {\rm O}_{2n+1}(R)$, then the matrices $\theta_{n+1}\cdot(\alpha\cdot F_i^j(X\cdot f)\cdot \alpha^{-1})\cdot \theta_{n+1}^{-1}$\\[5pt] and $\theta_{n+1}^{-1}\cdot(\alpha\cdot F_i^j(X\cdot f)\cdot \alpha^{-1})\cdot \theta_{n+1}$, lies in ${\rm EO}_{2n+1}(R[X])$,
    for $1\leq j\leq 2$ and $1 \leq i \leq n$.
\end{lemma}
\begin{proof}
    Let $v=\begin{pmatrix}
        v_0&v'&v''
    \end{pmatrix}^T$ and $w=\begin{pmatrix}
       w_0&w'&w'' 
    \end{pmatrix}^T$ be the $(i+1)^{\rm th}$ and $1^{\rm st}$ column of $\alpha$ respectively. Then we have, 
    \begin{align*}
\theta_{n+1}\cdot(\alpha\cdot F_i^1(f\cdot X)\cdot \alpha^{-1})\cdot \theta_{n+1}^{-1}&=\theta_{n+1}\cdot(\alpha\cdot{\rm E^o}_{e_{i+1},e_1(-1)}(f\cdot X) \cdot \alpha^{-1})\cdot \theta_{n+1}^{-1}\\[5pt]&=\theta_{n+1}\cdot({\rm E^o}_{v,w}(f\cdot X))\cdot \theta_{n+1}^{-1}\\[5pt]&= {\rm E^o}_{\theta_{n+1} \cdot v,\theta_{n+1}\cdot w}(f)\\[5pt]&={\rm E^o}_{\tiny\begin{pmatrix}
       v_0\cdot X\\ v'\cdot X\\v''\end{pmatrix},\begin{pmatrix}
          w_0\cdot X\\  w'\cdot X\\w''
        \end{pmatrix}}  (f)=\sigma(X).
        \end{align*}
       
        Let $v^\prime=\begin{pmatrix}
        v_1^\prime& \dots&v_n^\prime
    \end{pmatrix}^T$, $v''=\begin{pmatrix}
        v_1''& \dots&v_n''
    \end{pmatrix}^T$, $w'=\begin{pmatrix}
        w_1'& \dots&w_n'
    \end{pmatrix}^T$ and $w''=\begin{pmatrix}
        w_1''& \dots&w_n''
    \end{pmatrix}^T$, then we have
    $$\sigma(0)={\rm E^o}_{\tiny\begin{pmatrix}
            0\\O\\v''
        \end{pmatrix},\begin{pmatrix}
            0\\O\\w''
        \end{pmatrix}}(f(0))= \begin{pmatrix}
            1&O&O\\O&I_n&O\\O&f(0)\cdot (v''\cdot (w'')^T-w''\cdot(v'')^T)&I_n
        \end{pmatrix}.$$ By Theorem~\ref{Aux4.2}, it is clear that $\sigma(0)\in {\rm EO}_{2n+1}(R)$. Using Theorem~\ref{Aux4.14}, it is sufficient to consider the case when $R$ is a local ring. If $o(\begin{pmatrix}
            w_0\\w''
        \end{pmatrix})=R$, our assertion follows directly from  Theorem~\ref{Aux4.8}. If $o(w'')\neq R$, then there is a column $u=\begin{pmatrix}
            u_0&u'&u''
        \end{pmatrix}^T$ of the matrix $\alpha$ which is not the $1^{\rm st}$ or $(n+i+1)^{\rm th}$ column such that $o(u'')=R.$ For $a,b\in R$, by Lemma~\ref{Pre2.3}, we have$$\sigma={\rm E^o}_{\tiny\begin{pmatrix}
       a\cdot X\\ v'\cdot X\\v''\end{pmatrix},\begin{pmatrix}
          b\cdot X\\  w'\cdot X\\w''
        \end{pmatrix}}  (f)={\rm E^o}_{\tiny\begin{pmatrix}
       a\cdot X\\ v'\cdot X\\v''\end{pmatrix},\begin{pmatrix}
          c\cdot X\\  u'\cdot X\\u''
        \end{pmatrix}}  (f)\cdot {\rm E^o}_{\tiny\begin{pmatrix}
       a\cdot X\\ v'\cdot X\\v''\end{pmatrix},\begin{pmatrix}
          (b-c)\cdot X\\  (w'-u')\cdot X\\(w''-u'')
        \end{pmatrix}}  (f).$$ By Corollary~\ref{Aux4.9}, both the factors lies in ${\rm EO}_{2n+1}(R).$ Using the following derivation, $$\theta_{n+1}^{-1}\cdot (\alpha\cdot F_i^1(f\cdot X)\cdot\alpha^{-1})\cdot \theta_{n+1}=\sigma_\delta^{-1}\cdot(\theta_{n+1}^{-1}\cdot(\sigma_\delta\cdot\alpha)\cdot F_i^1(f\cdot X)\cdot (\sigma_\delta\cdot \alpha)^{-1})\cdot\theta_{n+1})\cdot\sigma_\delta,$$ we can obtain the second assertion. By performing similar calculations, we can prove the result for the generators $F_i^2(X\cdot f)$.      
\end{proof}
\noindent Next, we observe the following elementary fact in
group theory.
\begin{remark}\label{Ele5.2}
    Let $G$ be a group and $a_i,b_i \in G$ for $i = 1,...,n.$ Then
$$\displaystyle\prod _{i=1}^n a_i~b_i=\displaystyle\prod_{i=1}^n r_i~b_i~r_i^{-1}\cdot \displaystyle\prod_{i=1}^n a_i,~~
{\rm where}~~ r_i = \displaystyle\prod_{j=1}^i a_j.$$
\end{remark}
\begin{lemma}\label{Ele5.3}
   For $n\geq 3$, if $\ \beta\in {\rm{EO}_{2n+1}}(R[X])\bigcap{\rm{O}_{2n+1}}(R[X],X\cdot R[X])$, then the matrices $\theta_{n+1}\cdot\beta\cdot \theta_{n+1}^{-1}$\\[5pt] and $\theta_{n+1}^{-1}\cdot\beta\cdot\theta_{n+1}$ lies in ${\rm{EO}_{2n+1}}(R[X])$.
    \end{lemma}
    \begin{proof}
        We consider $\theta_{n+1}\cdot\beta\cdot \theta_{n+1}^{-1}$. Without loss of generality, let $\beta=\displaystyle\prod_{k=1}^{n}F_{i_k}^1(a_k+X\cdot f_k)\cdot F_{i_k}^2(a_k+X\cdot f_k)$, where $a_k\in R$ and $f_k\in R[X]$. Put $\alpha_p=\displaystyle\prod_{k=1}^{p}F_{i_k}^1(a_k)\cdot F_{i_k}^2(a_k).$ Then, clearly $\alpha_n=I_n$ and Lemma~\ref{Ele5.2} gives $\beta=\displaystyle\prod_{k=1}^{n}\alpha_p\cdot F_{i_k}^1(f_k\cdot X)\cdot F_{i_k}^2(f_k\cdot X)\cdot \alpha_p^{-1}$. Hence, it follows from  Lemma~\ref{Ele5.1}. Similarly, we can prove it for $\theta_{n+1}^{-1}\cdot\beta\cdot\theta_{n+1}$.
    \end{proof}
    \begin{lemma}\label{Ele5.4}
    Suppose $n\geq 3$, and $\beta\in {\rm{EO}_{2n+1}}(R[X])$. Then we have the following. 
    \begin{enumerate}[label=\emph{(\roman*)}]
        \item If $\theta_{n+1}\cdot\beta\cdot \theta_{n+1}^{-1}\in {\rm O}_{2n+1}(R[X])$, then $\theta_{n+1}\cdot\beta\cdot \theta_{n+1}^{-1}\in {\rm EO}_{2n+1}(R[X])$
        \item If $\theta_{n+1}^{-1}\cdot\beta\cdot \theta_{n+1}\in {\rm O}_{2n+1}(R[X])$, then $\theta_{n+1}^{-1}\cdot\beta\cdot \theta_{n+1}\in {\rm EO}_{2n+1}(R[X])$.
    \end{enumerate}
    \end{lemma}
    \begin{proof}
        It suffices to show the result for $\theta_{n+1}\cdot\beta\cdot \theta_{n+1}^{-1}$. Let $\beta=\beta_0\cdot\beta_1$, where $\beta_1\in {\rm EO}_{2n+1}(R)$ is the free\\[5pt] term of $\beta$. By Lemma~\ref{Ele5.3}, $\theta_{n+1}\cdot\beta_1\cdot \theta_{n+1}^{-1}\in {\rm EO}_{2n+1}(R[X])$. Since $\theta_{n+1}\cdot\beta\cdot \theta_{n+1}^{-1}\in {\rm O}_{2n+1}(R[X])$, $\beta_0$ has the form $$\beta_0=\begin{pmatrix}
            \alpha_{11}&\alpha_{12}&\alpha_{13}\\
            \alpha_{21}&\alpha_{22}&\alpha_{23}\\
            O&O&\alpha_{33}\end{pmatrix},$$
             where $\alpha_{11}\in {\rm M}_1(R)$, $\alpha_{12}$ and $\alpha_{13} \in {\rm M}_{1,n}(R)$, $\alpha_{21}\in {\rm M}_{n,1}(R)$, $\alpha_{22},\alpha_{23}$ and $\alpha_{33}\in {\rm M}_{n}(R).$ By the orthogonality, we have the following observations: $\alpha_{11}^2=1, \alpha_{12}=0$ and $(\alpha_{22}^T)^{-1}=\alpha_{33}.$
                
          \vspace{2mm}
          
            Hence we can write $\beta_0$ as follows, $$\beta_0=\begin{pmatrix}
            \alpha_{11}&O&\alpha_{13}\\
            \alpha_{21}&\alpha_{22}&\alpha_{23}\\
            O&O&(\alpha_{22}^T)^{-1}\end{pmatrix}.$$
            
            We can compute  $\theta_{n+1}\cdot\beta_0\cdot \theta_{n+1}^{-1}$ as follows,
            \begin{align*}
           \theta_{n+1}\cdot\beta_0\cdot \theta_{n+1}^{-1}&=\begin{pmatrix}
            \alpha_{11}&O&X\cdot\alpha_{13}\\
            \alpha_{21}&\alpha_{22}&X\cdot\alpha_{23}\\
            O&O&(\alpha_{22}^T)^{-1}\end{pmatrix}
            =\beta_0\cdot\begin{pmatrix}
                1&O&(X-1)\alpha_{11}\cdot\alpha_{13}\\
                O&I_n&(1-X)\alpha_{23}^T\cdot\alpha_{33}\\
                O&O&I_n
            \end{pmatrix}.
            \end{align*}
            The matrix $\begin{pmatrix}
                1&O&(X-1)\alpha_{11}\cdot\alpha_{13}\\
                O&I_n&(1-X)\alpha_{23}^T\cdot\alpha_{33}\\
                O&O&I_n
            \end{pmatrix}$ lies in ${\rm O}_{2n+1}(R[X])$, which implies that $\alpha_{13}=0$. Therefore, we have the relation $\alpha_{23}^T\cdot\alpha_{33}+\alpha_{33}^T\cdot\alpha_{23}=0$. That is, $\alpha_{23}^T\cdot\alpha_{33}\in {\rm Alt}_n(R)$. Hence, we get   $$\theta_{n+1}\cdot\beta_0\cdot \theta_{n+1}^{-1}=\beta_0\cdot\begin{pmatrix}
                1&O&O\\
                O&I_n&(1-X)\alpha_{23}^T\cdot\alpha_{33}\\
                O&O&I_n
            \end{pmatrix}\in {\rm EO}_{2n+1}(R[X]).$$
    \end{proof}
  \emph  {For the rest of this section, assume that $R$ is a local ring with
maximal ideal $\mathfrak{m}$. Let $k$ be the residue field and $\phi$ be the canonical projection,
$R\rightarrow k$}. 
    \begin{theorem}\label{Ele5.5}
        For $n\geq 3$, the following set inclusion holds, ${\rm O}_{2n+1}(\mathfrak m)\subset {\rm DO}_{2n+1}(\mathfrak m)\cdot {\rm EO}_{2n+1}(R,\mathfrak m).$
    \end{theorem}
    \begin{proof}
        We have the decomposition ${\rm O}_{2n+1}(R)={\rm O}_{2}(R)\cdot {\rm EO}_{2n+1}(R)$. Then we get, 
    ${\rm O}_{2n+1}(\mathfrak{m}) = {\rm O}_{2n+1}(R) \cap {\rm GL}_{2n+1}(\mathfrak{m}) = ({\rm O}_2(R) \cdot {\rm EO}_{2n+1}(R)) \cap {\rm GL}_{2n+1}(\mathfrak{m}) = {\rm O}_2(\mathfrak{m}) \cdot {\rm EO}_{2n+1}(\mathfrak{m})\subset {\rm DO}_{2n+1}(\mathfrak{m}) \cdot {\rm EO}_{2n+1}(R,\mathfrak{m})$.
    \end{proof}

 \begin{lemma}\label{Ele5.6}
 
     If $\alpha\in {\rm DO}_{2n+1}(R)$, then $\alpha \cdot {\rm TO}_{2n+1}(R)\cdot \alpha^{-1}={\rm TO}_{2n+1}(R).$
 \end{lemma} 
\begin{proof}
Let $\alpha\in {\rm DO}_{2n+1}(R)$, then $\alpha ={diag}(d_0,d_1,\dots ,d_n,d_1^{-1},\dots,d_{n}^{-1}).$  Clearly, for  $\alpha\in {\rm DO}_{2n+1}(R)$, ${\rm TO}_{2n+1}(R)\subset \alpha \cdot {\rm TO}_{2n+1}(R)\cdot \alpha^{-1}$.  
    By computations, we get the following relations for the generators of ${\rm TO}_{2n+1}(R)$; 
    $\alpha\cdot F_i^1(z)\cdot\alpha^{-1}=F_i^1(z\cdot d_i)$ and  $\alpha\cdot F_i^2(\frac{1}{2})\cdot\alpha^{-1}=F_i^2(\frac{1}{2}\cdot d_i^{-1})$, for $1\leq i\leq n$.
     Thus, we obtain $\alpha \cdot {\rm TO}_{2n+1}(R)\cdot \alpha^{-1}\subset {\rm TO}_{2n+1}(R)$.  
   
\end{proof}

    \begin{theorem}\label{Ele5.7}
        The orthogonal group ${\rm O}_{2n+1}(R)$ has the following decomposition, ${\rm O}_{2n+1}(R)={\rm TO}_{2n+1}(R)\cdot{\rm MO}_{2n+1}(R)\cdot{\rm TO}_{2n+1}(R)\cdot{\rm EO}_{2n+1}(R,\mathfrak m).$
    \end{theorem}
    \begin{proof}
        Let $\alpha \in {\rm O}_{2n+1}(R)$, then by Theorems ~\ref{Aux4.15} and \ref{Aux4.17}, there exist $\beta_1, \beta_3\in {\rm TO}_{2n+1}(k)$, $\beta_2\in {\rm MO}_{2n+1}(k)$, and $\alpha_1,\alpha_3\in {\rm TO}_{2n+1}(R)$, $\alpha_2\in {\rm MO}_{2n+1}(R)$ such that $\phi(\alpha)=\beta_1\cdot\beta_2\cdot\beta_3$ and $\phi(\alpha_i)=\beta_i$. Thus, $\alpha=\alpha_1\cdot\alpha_2\cdot\alpha_3\cdot\alpha_4$, where $\alpha_4\in {\rm O}_{2n+1}(R) $. Writing $\alpha_4$ in the form $\alpha_5\cdot\alpha_6$, where $\alpha_5\in {\rm DO}_{2n+1}(R)$ and $\alpha_6 \in {\rm EO}_{2n+1}(R,I)$, we obtain $ \alpha =\alpha_1\cdot(\alpha_2\cdot\alpha_5)\cdot(\alpha_5^{-1}\cdot\alpha_3\cdot\alpha_5)\cdot\alpha_6$, where $\alpha_1\in{\rm TO}_{2n+1}(R)$ and $\alpha_6\in{\rm EO}_{2n+1}(R,\mathfrak m)$. Moreover, one can easily verified that $\alpha_2\cdot\alpha_5$ is in ${\rm MO}_{2n+1}(R).$ Using Lemma~\ref{Ele5.6}, $\alpha_5^{-1}\cdot\alpha_3\cdot\alpha_5$ lies in ${\rm TO}_{2n+1}(R)$. 
    \end{proof}
    Define a subset $V$ of ${\rm O}_{2n+1}(R[X,X^{-1}]$ as $V=V_1\cdot V_2\cdot V_3\cdot V_4$, where \begin{multicols}{2}
    \begin{itemize}
        \item 
    $V_1:={\rm O}_{2n+1}(R)\cdot {\rm EO}_{2n+1}(R[X])$,
 \item $V_2:=\{diag (1,X^{K_1},\dots,X^{K_n},X^{-K_1},\dots,X^{-K_n})\}$,
 \item $V_3:={\rm TO}_{2n+1}(R[X,X^{-1}])$,\item $V_4:={\rm EO}_{2n+1}(R[X,X^{-1}],\mathfrak {m}[X,X^{-1}])$
 \end{itemize}
 \end{multicols}

\begin{lemma}\label{Ele5.8}
    We have the relations $\theta_{n+1}\cdot V \cdot\theta_{n+1}^{-1} \subset V$ and $\theta_{n+1}^{-1}\cdot V\cdot\theta_{n+1} \subset V$.
\end{lemma}
\begin{proof}
    $V_4$ is a normal subgroup of ${\rm O}_{2n+1}(R[X,X^{-1}])$ and $\theta_{n+1}$ normalizes the subgroup $V_2\cdot V_3$. It is enough to show that, for $\alpha\in V_1,\ \theta_{n+1}\cdot \alpha \cdot {\theta^{-1}_{n+1}}\in V.$ 
    \vspace{1mm}
    
    We have  $\alpha=\alpha_o\cdot \alpha_1$, where $\alpha_o\in {\rm O}_{2n+1}(R)$ and $\alpha_1\in {\rm EO}_{2n+1}(R[X])\cap {\rm O}_{2n+1}(R[X],X\cdot R[X])$. Hence, using Lemma~\ref{Ele5.3}, it suffices to prove that $\theta_{n+1}\cdot \alpha \cdot {\theta^{-1}_{n+1}}\subset V$, for $\alpha\in {\rm O}_{2n+1}(R)$. By Theorem~\ref{Ele5.7}, the matrix $\alpha$ can be written in the form $\alpha=\alpha_1\cdot\alpha_2\cdot\alpha_3\cdot\alpha_4\cdot\alpha_5$, where $\alpha_1,\alpha_4\in {\rm TO}_{2n+1}(R)$, $\alpha_2\in {\rm DO}_{2n+1}(R)$, $\alpha_3\in {\rm \Pi O}_{2n+1}(R)$, and $\alpha_5\in {\rm EO}_{2n+1}(R,\mathfrak{m})$.
     Now, we observe the following, \begin{enumerate}[label=(\roman*)]
        \item $\theta_{n+1}\cdot \alpha_1 \cdot {\theta^{-1}_{n+1}} \in {\rm O}_{2n+1}(R[X]$, hence in ${\rm EO}_{2n+1}(R[X])$,
        \item $[\alpha_3,\theta_{n+1}] \in V_2$,
        \item $\theta_{n+1}\cdot \alpha_4 \cdot {\theta^{-1}_{n+1}}\in {\rm TO}_{2n+1}(R[X,X^{-1}])$, and
        \item $\theta_{n+1}\cdot \alpha_5 \cdot {\theta^{-1}_{n+1}} \in {\rm EO}_{2n+1}(R[X,X^{-1}],\mathfrak {m}[X,X^{-1}])$.
    \end{enumerate}
     Thus, $\theta_{n+1}\cdot \alpha \cdot {\theta_{n+1}}=(\theta^{-1}_{n+1}\cdot \alpha_1 \cdot {\theta^{-1}_{n+1}}\cdot \alpha_2\cdot\alpha_3)\cdot [\alpha_3^{-1},\theta_{n+1}]\cdot (\theta_{n+1}\cdot\alpha_4\cdot\theta_{n+1}^{-1})\cdot(\theta_{n+1}\cdot\alpha_5\cdot\theta_{n+1}^{-1})\in V $.
\end{proof} 
\begin{corollary}\label{Ele5.9}
    For any $i\neq j$, and $a\in R$, $F_i^1(a\cdot X)\cdot V\subset V$,  $F_i^2(a\cdot X)\cdot V\subset V$, $F_i^1(a\cdot X^{-1})\cdot V\subset V$, and $F_i^2(a\cdot X^{-1})\cdot V\subset V$.
\end{corollary}
\begin{proof}
    Clearly, $F_i^1(a\cdot X)\cdot V\subset V$ and $F_i^2(a\cdot X)\cdot V\subset V$. For the reverse inclusion, we have the existence of a permutation $\pi$ such that $\pi(i)$ and $\pi\delta(j)$ lie in the set $\{1,\dots,n\}$. Since $\sigma_\pi\cdot V\subset V$ and $\sigma_\pi\cdot {\rm E}^o_{i,j}(a\cdot X^{-1})\cdot \sigma_\pi^{-1}= {\rm E}^o_{\pi(i),\pi(j)}(a\cdot X^{-1})$, for $1\leq i\leq n$ and $n+1\leq j\leq 2n$, we have $ {\rm E}^o_{i,j}(a\cdot X^{-1})=\theta_{n+1}^{-1}\cdot  {\rm E}^o_{i,j}(a)\cdot \theta_{n+1}$ and $ {\rm E}^o_{i,j}(a\cdot X^{-1})\cdot V=\theta_{n+1}^{-1}\cdot ({\rm E}^o_{i,j}(a)\cdot (\theta_{n+1}\cdot V\cdot \theta_{n+1}^{-1}))\cdot \theta_{n+1}\subset V$.
\end{proof}
    \begin{corollary}\label{Ele5.10}
     The elementary orthogonal ${\rm EO}_{2n+1}(R[X,X^{-1}])$ is contained in the set $V$.
    \end{corollary}
    \begin{proof}
    The group ${\rm EO}_{2n+1}(R[X,X^{-1}])$ is generated by matrices of the form $F_i^1(a\cdot X)$, $F_i^2(a\cdot X)$, $F_i^1(a\cdot X^{-1})$, and $F_i^2(a\cdot X^{-1})$. Hence, the result follows directly from Corollary~\ref{Ele5.9}.
\end{proof}

    \begin{theorem}\label{Ele5.11}
        If $\alpha \in {\rm EO}_{2n+1}(R[X,X^{-1}])\cap {\rm O}_{2n+1}(R[X,X^{-1}],{\mathfrak{m}}[X,X^{-1}])$, then $\alpha = \alpha_1\cdot\alpha_2$ for certain elements $\alpha_1\in {\rm EO}_{2n+1}(R[X],{\mathfrak{m}}[X])$ and $\alpha_2\in {\rm EO}_{2n+1}(R[X,X^{-1}],{\mathfrak{m}}[X,X^{-1}])$.
    \end{theorem}
    \begin{proof}
    Let $\alpha \in {\rm EO}_{2n+1}(R[X,X^{-1}])$. Then $\alpha\in V$ and $\alpha=\beta_1\cdot\beta_2\cdot\beta_3\cdot\beta_4$,  where $\beta_i\in V_i$. Since $\alpha \equiv I_{2n+1}~{\rm mod }(\mathfrak{m}[X,X^{-1}])$, we have $(\phi(\beta_1))^{-1}=\phi(\beta_2)\cdot\phi(\beta_3)$. If $diag (1,X^{K_1},\dots,X^{K_n},X^{-K_1},\dots,X^{-K_n})=\phi(\beta_2)$, then $1$, $X^{K_1}$,$\dots,X^{-K_n}$ are equal to the diagonal elements of the matrix $(\phi(\beta_1))^{-1}=\phi(\beta_2)\cdot\phi(\beta_3)$. Consequently, we get $X^{\pm k_i} \in k[X]$ (Since the inverse of $\phi(\beta_1)$ will lie in ${\rm EO}_{2n+1}(k[X])\cdot {\rm O}_{2n+1}(k)$), hence $k_i=0$ and $\beta_2=I_{2n+1}$. Also, $\phi(\beta_3)=(\phi(\beta_1))^{-1} \in {\rm TO}_{2n+1}(k[X])$. Hence, by Theorem~\ref{Aux4.15}, there exists $\gamma \in {\rm TO}_{2n+1}(R[X]) $ such that $\phi(\gamma)=\phi(\beta_3)$. By Theorem~\ref{Aux4.15} (iii), $\alpha_2=(\gamma^{-1}\cdot\beta_3)\cdot \beta_4$ lies in ${\rm EO}_{2n+1}(R[X,X^{-1}],{\mathfrak{m}}[X,X^{-1}])$, and $\alpha_1=\beta_1\cdot\gamma$ lies in ${\rm O}_{2n+1}(R)\cdot {\rm EO}_{2n+1}(R[X])\cap {\rm EO}_{2n+1}(R[X,X^{-1}])={\rm EO}_{2n+1}(R[X])$.
    \end{proof}
\section{Horrocks' theorem}
In this section, we prove the main result of this article, that is Horrocks' theorem. We prove additional results that lead us to this theorem.
Throughout this section, assume that $R$ is a local ring (unless stated otherwise) with maximal ideal $\mathfrak{m}$ and residue file $k$. 

\vspace{1mm}

 Set $G_{\pm}= {\rm EO}_{2n+1}(R[X^{\pm1}])\cap {\rm O}_{2n+1}(m[X^{\pm1}]) \subset {\rm EO}_{2n+1}(R[X,X^{-1}])\cap {\rm O}_{2n+1}(m[X,X^{-1}])$, and define $ G=G_+\cdot G_-$. Let $H$ denote the subset of ${\rm O}_{2n+1}(R[X,X^{-1}])$ consisting of all elements $h$ such that $h\cdot G \cdot h^{-1}\in G$.

\begin{lemma}\label{Hor6.1}
    \begin{enumerate}[label=\emph{(\roman*)}]
        \item The set $H$ is multiplicatively closed.
        \item $ {\rm EO}_{2n+1}(R)\subset H.$
        \item If $n\geq 3$, then $\theta_{n+1},\theta_{n+1}^{-1} \in H.$
    \end{enumerate}
    \end{lemma}
    \begin{proof}
          Let $\alpha, \beta \in H$. Then $(\alpha\cdot \beta) \cdot G \cdot (\alpha\cdot\beta)^{-1}=\alpha \cdot (\beta\cdot G\cdot \beta^{-1})\alpha^{-1}\in G$,  which proves (i). Assertion (ii) is a consequence of the normality of the elementary orthogonal group ${\rm EO}_{2n+1}(R)$. 
          
          \vspace{1mm}
          
           Let $\alpha\in G$. Then $\alpha=\alpha_+\cdot \alpha_o\cdot \alpha_-$, where $\alpha_{\pm}\in {\rm EO}_{2n+1}(R[X^{\pm1}])\cap {\rm O}_{2n+1}(X^{\pm1},\mathfrak{m}[X^{\pm1}])$ and $\alpha_o\in {\rm E}_{2n+1}(R)\cap {\rm O}_{2n+1}(\mathfrak{m})$. Here, $\alpha_o$ can be written in the form  $\alpha_o=\delta_1 \cdot \delta_2 \cdot \delta_3\cdot\delta_4 $, where $\delta_1,\delta_3\in {\rm TO}_{2n+1}(R)$, $\delta_2\in {\rm MO}_{2n+1}(R)$ and $\delta_4\in {\rm EO}_{2n+1}(R,\mathfrak{m})$. 
          Then, we get $\theta_n\cdot \alpha \cdot \theta_n^{-1}=\theta_n\cdot  \alpha_+\cdot \alpha_o\cdot \alpha_-\cdot \theta_n^{-1}=(\theta_n\cdot\alpha_+\cdot\delta_1\cdot\theta_n^{-1})\cdot\delta_2\cdot(\theta_n\cdot\delta_3\cdot\delta_4\cdot\alpha_-\cdot\theta_n^{-1})$.
          The first factor belongs to $G_+$, by Lemma~\ref{Ele5.4}, and the last factor belongs to $G_-$, by an analogue of Lemma~\ref{Ele5.4}. Also, we have  $\delta_2\in G_+\cap G_-$, hence $\theta_n\in H$. Moreover, by (i) and (ii), we have $\theta_n^{-1}=(\sigma_\delta\cdot\theta_n\cdot\sigma_\delta^{-1})\cdot X^{-1}\in H$.
       
    \end{proof}

 \begin{corollary}\label{Hor6.2}
     If $a\in R$ and $i\neq j$, then $F_i^1(a\cdot X)$, $F_i^2(a\cdot X)$, $F_i^1(a\cdot X^{-1})$, $F_i^2(a\cdot X^{-1})\in H$.
 \end{corollary}
 \begin{proof}
     There exist a permutation $\pi$ which commutes with $\delta$ such that $\pi(i)$ and $\pi(\delta(i))$ lies between 1 and $n$. By Lemma~\ref{Hor6.1}, it is easy to verify that $F_{\pi(i)}^1(a\cdot X)=\theta_{n+1} \cdot F_{\pi(i)}^1(a) \cdot {\theta_{n+1}}^{-1}\in H$, since $F_{\pi(i)}^1(a)\in H$. We obtain $F_i^1(a\cdot X)=\sigma_\pi^{-1}\cdot F_{\pi(i)}^1(a\cdot X)\cdot \sigma_\pi\in H$, by Lemma~\ref{Hor6.1}.
 \end{proof}
    \begin{theorem}\label{Hor6.3}
        The subgroup ${\rm EO}_{2n+1}(R[X,X^{-1}])$ is contained in $H$.
    \end{theorem}
\begin{proof}
    Since these four generators, $F_i^1(a\cdot X), F_i^2(a\cdot X), F_i^1(a\cdot X^{-1})$ and $F_i^2(a\cdot X^{-1})$ generate the group ${\rm EO}_{2n+1}(R[X,X^{-1}])$, from Lemma~\ref{Hor6.1} and Corollary~\ref{Hor6.2}, we have the result.
\end{proof}

    \begin{lemma}\label{Hor6.4} Let $K$ be the subset of ${\rm EO}_{2n+1}(R[X,X^{-1}])\cap {\rm O}_{2n+1}({\mathfrak m}[X,X^{-1}])$, consisting of all $x$ such that $x\cdot G \subset G.$ Then we have the following.
    \begin{enumerate}[label=\emph{(\roman*)}]
            \item $K$ is closed under multiplication.
            \item $K \supset {\rm EO}_{2n+1}(R[X])\cap {\rm O}_{2n+1}({\mathfrak m}[X]) = G_+$.
            \item $K\supset {\rm EO}_{2n+1}(R[X,X^{-1}],{\mathfrak m}[X,X^{-1}]).$
        \end{enumerate}
    \end{lemma}
    \begin{proof}
    Assertion (i) and (ii) can be verified easily. Let $\beta\in {\rm EO}_{2n+1}(\mathfrak{m}[X]),\gamma\in {\rm EO}_{2n+1}(\mathfrak{m}[X^{-1}])$ and $\delta\in {\rm EO}_{2n+1}(R[X,X^{-1}])$. Then, we have $\delta^{-1}\cdot \beta \cdot \delta\cdot G\subset \delta^{-1}\cdot \beta \cdot G \cdot \delta\subset \delta^{-1}\cdot G\cdot \delta\subset G ~{\rm and}~ \delta^{-1}\cdot \gamma \cdot \delta\cdot G\subset \delta^{-1}\cdot \gamma \cdot G \cdot \delta\subset \delta^{-1}\cdot G\cdot \delta\subset G.$ Thus, we get $\delta^{-1}\cdot \beta \cdot \delta\in K$ and $\delta^{-1}\cdot \gamma \cdot \delta\in K$. Also, ${\rm EO}_{2n+1}(R[X,X^{-1}],\mathfrak{m}[X,X^{-1}])$ is generated by the matrices of the form $\delta^{-1}\cdot \beta \cdot \delta$ and $\delta^{-1}\cdot \gamma \cdot \delta$. Hence by (i), we have (iii).
\end{proof}

    \begin{theorem}\label{Hor6.5}
        For $n\geq 3$, any matrix $\alpha \in {\rm EO}_{2n+1}(R[X,X^{-1}])\cap {\rm O}_{2n+1}({\mathfrak m}[X,X^{-1}])$ can be represented in the form $\alpha_+\cdot \alpha_-$, where $\alpha_\pm \in {\rm EO}_{2n+1}(R[X^{\pm1}])\cap {\rm O}_{2n+1}({\mathfrak m}[X^{\pm1}]).$
    \end{theorem}

\begin{proof}
    Let $\alpha\in {\rm EO}_{2n+1}(R[X,X^{-1}])\cap {\rm O}_{2n+1}({\mathfrak m}[X,X^{-1}])$. Using Theorem~\ref{Ele5.11}, we have $\alpha=\alpha_1\cdot\alpha_2$ for $\alpha_1\in {\rm EO}_{2n+1}(R[X],{\mathfrak{m}}[X])$ and $\alpha_2\in {\rm EO}_{2n+1}(R[X,X^{-1}],{\mathfrak{m}}[X,X^{-1}])$. Since both factors lie in $K$, the product also lies in $K$, by Lemma~\ref{Hor6.4}. Using definition of $K$, we have the splitting $\alpha=\alpha_+\cdot \alpha_-$ where $\alpha_\pm \in {\rm EO}_{2n+1}(R[X^{\pm1}])\cap {\rm O}_{2n+1}({\mathfrak m}[X^{\pm1}])$. 
\end{proof}
   \begin{theorem}\label{Hor6.6} {$($Horrocks theorem$)$}
      Suppose $R$ is a commutative ring in which 2 is invertible (need not be local). Let $\alpha \in {\rm O}_{2n+1}(R[X])$ and $\beta \in {\rm O}_{2n+1}(R[X^{-1}])$, for $n\geq 3$. If $\alpha \cdot \beta^{-1}\in {\rm EO}_{2n+1}(R[X,X^{-1}])$, then $\alpha \in {\rm O}_{2n+1}(R)\cdot{\rm EO}_{2n+1}(R[X])$ and $\beta \in {\rm O}_{2n+1}(R) \cdot {\rm EO}_{2n+1}(R[X^{-1}])$.
   \end{theorem}
\begin{proof}
Using Theorem~\ref{Aux4.14}, it is sufficient to consider the case of a local ring. Assume $R$ is a local ring with maximal ideal $\mathfrak{m}$. 
Let $k$ be the quotient field $R/\mathfrak m$ and $\phi$ 
be the canonical homomorphism from $R$ to $k$.
By Theorem~\ref{Aux4.12}, multiply $\alpha$ and $\beta$ by matrices in ${\rm EO}_{2n+1}(R[X])$ and ${\rm EO}_{2n+1}(R[X^{-1}])$ respectively.
\vspace{1mm}

Assume that, $\phi(\alpha) \in {\rm O}_{2n+1}(k[X]) 
\subset {\rm MO}_{2n+1}(k)$ 
and $\phi(\beta) \in {\rm O}_{2n+1}(k[X^{-1}]) \subset {\rm MO}_{2n+1}(k)$. Since $\alpha \cdot \beta^{-1} \in {\rm EO}_{2n+1}(R[X,X^{-1}]$, it follows that $\phi(\alpha) \cdot {\phi(\beta)}^{-1} \in {\rm EO}_{2n+1}(k[X,X^{-1}])$. Using Lemma~\ref{Aux4.11}, we get $\phi(\alpha) \cdot (\phi(\beta))^{-1} \in {\rm EO}_{2n+1}(k)$. Then, there exists $\alpha_1 \in {\rm MO}_{2n+1}(R) $ such that $\phi(\alpha_1)=\phi(\alpha)$, by Theorem~\ref{Aux4.15} and there exists $\gamma \in {\rm EO}_{2n+1}(R) $ such that $\phi(\gamma)=\phi(\alpha)\cdot \phi(\beta)^{-1}$. 

Consider the matrix $(\alpha_1^{-1}\cdot \alpha)\cdot(\alpha_1^{-1}\cdot\gamma\cdot\beta)^{-1}=\alpha_1^{-1}\cdot(\alpha\cdot\beta^{-1}\cdot\gamma^{-1})\cdot\alpha_1$. By normality of DSER group (Theorem~\ref{Aux4.10}), this matrix lies in ${\rm EO}_{2n+1}(R[X,X^{-1}])$. Therefore, it suffices to prove that the matrices $\alpha_1^{-1}\cdot \alpha\in {\rm O}_{2n+1}(R) \cdot {\rm EO}_{2n+1}(R[X])$ and the matrix $\alpha_1^{-1}\cdot\gamma\cdot\beta \in {\rm O}_{2n+1}(R)\cdot {\rm EO}_{2n+1}(R[X^{-1}])$. Without loss of generality, we can assume that $\phi(\alpha)=I_{2n+1}$ and $\phi(\beta)=I_{2n+1}$. Hence, we get $\alpha \cdot\beta^{-1} \in {\rm EO}_{2n+1}(R[X,X^{-1}])\cap {\rm O}_{2n+1}({\mathfrak m}[X,X^{-1}])=G$. Thus, by Theorem~\ref{Hor6.5}, we get $\alpha\cdot\beta^{-1}=\alpha_+\cdot \alpha_-$, where $\alpha_{\pm}\in G_\pm$. We have $(\alpha_+)^{-1}\cdot\alpha=\alpha_-\cdot\beta\in{\rm O}_{2n+1}(R[X])\cap{\rm O}_{2n+1}(R[X^{-1}])$, which is  the group ${\rm O}_{2n+1}(R)$. Hence, we have both $\alpha \in {\rm O}_{2n+1}(R) \cdot{\rm EO}_{2n+1}(R[X])$ and $\beta \in {\rm O}_{2n+1}(R) \cdot {\rm EO}_{2n+1}(R[X^{-1}])$.
  
  \end{proof}

\section{Acknowledgement} The first author gratefully acknowledges the support of the KSCSTE Young Scientist Award Scheme (2021-KSYSA-RG), Government of Kerala, for providing the grant that enabled this research. She also expresses her gratitude to SERB for the SURE grant (SUR/2022/004894) and to RUSA (RUSA 2.0-T3A), Government of India. The second author is thankful for the support received from the UGC Junior Research Fellowship Scheme (Ref. No. 221610053453). The authors thank Dr. Aparna Pradeep V. K. and Gayathry Pradeep for their valuable suggestions and contributions to the discussions.

        \bibliographystyle{amsplain}

   \end{document}